\documentclass[11pt,a4paper, british]{amsart}

\usepackage{babel}
\usepackage[useregional]{datetime2}
\DTMlangsetup[]{en-GB}

\usepackage{amssymb}
\usepackage[hidelinks]{hyperref}

\usepackage{tikz}
\usepackage{caption}
\usetikzlibrary{arrows}

\pgfmathsetmacro{\myxlow}{-2}
\pgfmathsetmacro{\myxhigh}{2}
\pgfmathsetmacro{\myiterations}{3}

\usepackage{tikz-cd}

\newtheorem{thm}{Theorem}[section]
\newtheorem{prop}[thm]{Proposition}
\newtheorem{lem}[thm]{Lemma}
\newtheorem{cor}[thm]{Corollary}

\theoremstyle{definition}
\newtheorem{definition}[thm]{Definition}

\theoremstyle{remark}
\newtheorem{remark}[thm]{Remark}

\numberwithin{equation}{section}

\newcommand{\G}{\mathbb{G}_\mathrm{m}}  
\newcommand{\alg}{\overline{\mathbb{Q}}} 
\newcommand{\R}{\mathbb{R}}  
\newcommand{\Q}{\mathbb{Q}} 
\newcommand{\Z}{\mathbb{Z}} 
\newcommand{\C}{\mathbb{C}} 
\newcommand{\h}{\mathbb{H}} 
\newcommand{\SL}{\mathrm{SL}_2(\Z)} 
\newcommand{\Gal}{\mathrm{Gal}} 
\newcommand{\AN}{\mathbb{A}^n_\C}

\DeclareMathOperator{\cl}{cl} 

\begin{document}
	
	\title{A uniform effective André--Oort result}
	\author{Guy Fowler}
	\address{Department of Mathematics, University of Manchester, Manchester, UK
		\newline
		\indent
		Heilbronn Institute for Mathematical Research, Bristol, UK}
	\email{\href{mailto:guy.fowler@manchester.ac.uk}{guy.fowler@manchester.ac.uk}}
	\urladdr{\url{https://www.guyfowler.uk/}}
	\date{\today}
	\subjclass[2020]{11G18, 14G35}
	
		\begin{abstract}
				We prove the André--Oort conjecture for hypersurfaces $V \subset Y(1)^n \cong \mathbb{A}^n_\C$ defined by an equation $a_1 x_1^m + \ldots + a_n x_n^m = b$, where $a_1, \ldots, a_n, b \in \alg$ and $m \in \Z_{>0}$.
				Unlike previous proofs, our result is both effective and uniform in the height of the coefficients $a_1, \ldots, a_n, b$.
				This is the first effective proof of a uniform André--Oort statement for a class of subvarieties with arbitrary dimension and non-empty special locus.
				We also prove an analogous result for hypersurfaces $V \subset Y(1)^n \times \G^l$.
			\end{abstract}
	
	\maketitle
	
	\section{Introduction}
	
	Let $X$ be a Shimura variety.
	Then $X$ has an associated collection of \textit{special subvarieties}.
	The Andr\'e--Oort conjecture, now a theorem \cite{PilaShankarTsimerman21}, states that, given a subvariety $V \subset X$, the \textit{special locus} of $V$ (i.e.~the union of all the special subvarieties contained in $V$) is equal to a finite union of special subvarieties.
	In particular, if $V^\mathrm{sp}$ denotes the union of all the positive-dimensional special subvarieties contained in $V$, then the set $V \setminus V^\mathrm{sp}$ contains only finitely many \textit{special points}, i.e.~special subvarieties of dimension $0$.
	
	In this article, we consider the case where $X = Y(1)^n$, a Cartesian power of the modular curve $Y(1) = \SL \backslash \h$.
	We identify $Y(1)^n$ with $\mathbb{A}^n_\C$ by means of the isomorphism given by the $j$-invariant.
	A special point is then a point $(x_1, \ldots, x_n) \in \mathbb{A}^n_\C$ for which $x_1, \ldots, x_n$ are all \textit{singular moduli}, i.e.~$j$-invariants of elliptic curves with complex multiplication (CM).
	The positive-dimensional special subvarieties of $\AN$ are described in Section~\ref{sec:specials}.
	
	The \textit{discriminant} of a singular modulus is defined to be the discriminant of the endomorphism order of a corresponding CM elliptic curve.
The discriminant is a negative integer. 
There are only finitely many singular moduli of a given discriminant, and these may be found effectively.
	
	Pila \cite{Pila11} proved the Andr\'e--Oort conjecture for $\AN$.
	The $n = 2$ case had been proved earlier by Andr\'e \cite{Andre98}.
	Pila's proof applied the Pila--Zannier strategy \cite{PilaZannier08} of o-minimal point counting.
Notably, this approach yields a very uniform version of Andr\'e--Oort, as illustrated by the following result.
	
	\begin{thm}[{\cite[Theorem~13.2]{Pila11}}]\label{thm:Pila}
		Let $m, n, d \in \Z_{>0}$.
		There exists a constant $c(m ,n, d)$ with the following property.
		Let $V \subset \AN$ be a hypersurface of degree $m$ defined over a number field $K$ with $[K : \Q] \leq d$.
		If $(x_1, \ldots, x_n) \in V \setminus V^\mathrm{sp}$ is a special point and $\Delta_i$ denotes the discriminant of $x_i$, then
		\[ \max_{1 \leq i \leq n} \lvert \Delta_i \rvert \leq c(m, n, d).\]
	\end{thm}

The constant $c(m, n, d)$ in Theorem~\ref{thm:Pila} is not effectively computable.
Indeed, the ineffectivity of the Landau--Siegel lower bound \cite{Landau35, Siegel35} for the class number of an imaginary quadratic field poses (see e.g.~\cite[p.~55]{Pila22}) a serious obstacle to proving the Andr\'e--Oort conjecture effectively via the Pila--Zannier strategy, even in light of recent progress \cite{Binyamini19a, Binyamini24} towards effective versions of the Pila--Wilkie counting theorem \cite{PilaWilkie06}.

\subsection{Main results}

The first main result of this paper is an effective version of Theorem~\ref{thm:Pila} for a certain class of subvarieties.
Notably this class includes subvarieties of arbitrary dimension, degree, and degree of field of definition.

\begin{thm}\label{thm:main}
	Let $m, n, d \in \Z_{>0}$.
	Let $V \subset \AN$ be a hypersurface defined by an equation
	\[ a_1 x_1^m + \ldots + a_n x_n^m = b,\]
	where $a_1, \ldots, a_n, b \in \alg$ are such that
	\[ [\Q(a_1, \ldots, a_n, b) : \Q] \leq d.\]
	If $(x_1, \ldots, x_n) \in V \setminus V^\mathrm{sp}$ is a special point and $\Delta_i$ denotes the discriminant of $x_i$, then
	\[ \max_{1 \leq i \leq n} \lvert \Delta_i \rvert^{1/2} \leq \exp(10^{33} 10^{13n} (d!)^{12} m^4).\]
\end{thm}

Significantly, the bound in Theorem~\ref{thm:main} is entirely independent of the height of the coefficients $a_1, \ldots, a_n, b$.
We will also prove similarly uniform explicit bounds on the positive-dimensional maximal special subvarieties of $V$, see Theorem~\ref{thm:positive}.

Our second main result generalises Theorem~\ref{thm:main} to linear dependencies among powers of singular moduli and roots of unity.
To state our result, we make the following definition, along the lines of \cite[Definition~7.6]{Pila14a}.

\begin{definition}\label{def:tuples}
	Let $m, d \in \Z_{> 0}$ and let $n, l \in \Z_{\geq  0}$.
	 An $(n+l)$-tuple $(x_1, \ldots, x_n, \zeta_1, \ldots, \zeta_l) \in \C^{n+l}$ is called a \textit{non-degenerate} $(m, d, n, l)$-\textit{tuple} if:
	 \begin{enumerate}
	 	\item $x_1, \ldots, x_n$ are singular moduli;
	 	\item $\zeta_1, \ldots, \zeta_l$ are roots of unity; and
	 	\item there exist $a_1, \ldots, a_n, b_1, \ldots, b_l, c \in \alg$ such that:
	 	\begin{enumerate}
	 		\item $a_1 x_1^m + \ldots + a_n x_n^m + b_1 \zeta_1 + \ldots + b_l \zeta_l + c = 0$;
	 		\item $ [\Q(a_1, \ldots, a_n, b_1, \ldots, b_l, c) : \Q] \leq d$;
	 		\item there does not exist a non-empty subset $I \subset \{1, \ldots, n\}$ such that $\# \{x_i : i \in I\} = 1$ and $\sum_{i \in I} a_i = 0$; and
	 		\item no non-empty proper sub-sum of $b_1 \zeta_1 + \ldots + b_l \zeta_l + c$ vanishes (but $c=0$ is allowed if $l = 0$).
	 	\end{enumerate}
	 \end{enumerate}  
\end{definition}

\begin{thm}\label{thm:tups}
	Let $m, d \in \Z_{> 0}$ and let $n, l \in \Z_{\geq  0}$.
	If $(x_1, \ldots, x_n, \zeta_1, \ldots, \zeta_l)$ is a non-degenerate $(m, d, n, l)$-tuple, $\Delta_i$ denotes the discriminant of the singular modulus $x_i$, and $N_j$ denotes the order of the root of unity $\zeta_j$, then
	\begin{align*}
		\max_{1 \leq i \leq n} \lvert \Delta_i \rvert^{1/2} &\leq \exp( (l+1)^8 \exp(10^{22} 10^{27n} (d!)^{13} m^4)),\\
		\max_{1 \leq j \leq l} N_j &\leq \exp((l+1)^8 \exp(10^{24} 10^{27n} (d!)^{13} m^4)).
	\end{align*}
\end{thm}

Pila \cite[Theorem~7.7]{Pila14a} gave an ineffective proof of the $m = 1$ case of Theorem~\ref{thm:tups}.
More generally, his ineffective proof \cite[Theorem~1.1]{Pila11} of Andr\'e--Oort for $\AN$ also covers subvarieties of $\AN \times \G^l$.
In this mixed setting, ``special points'' are points of $\AN \times \G^l$ whose $\AN$ coordinates are singular moduli and whose $\G^l$ coordinates are roots of unity.
Paulin \cite{Paulin15, Paulin16} proved an effective Andr\'e--Oort result for curves $V \subset \mathbb{A}^1_\C \times \G$; in particular, his result implies \cite[Remark, p.~162]{Paulin15} the $n = l = 1$ case of Theorem~\ref{thm:tups}.
 
When $n = 0$ and $d = 1$, Theorem~\ref{thm:tups} corresponds to a classical result of Mann \cite{Mann65} on roots of unity that are linearly dependent over $\Q$.
The full $n = 0$ case is due to Schinzel \cite{Schinzel88}, subsequently improved by Dvornicich and Zannier \cite{DvornicichZannier00};
see also results of Schlickewei \cite{Schlickewei96} and Evertse \cite{Evertse99}.
These results are instances of the Manin--Mumford conjecture for $\G^n$, proved originally by Laurent \cite{Laurent84}, with refined and uniform versions proved subsequently by many others, see e.g.~\cite{Martinez19}.

\subsection{Previous effective results on André--Oort}

To the author's knowledge, the only cases prior to Theorem~\ref{thm:main} where such a uniform version of Andr\'e--Oort for $\AN$ was known effectively were:
\begin{enumerate}
	\item when $V \subset \mathbb{A}^2_\C$ is a curve with Zariski-closure in $\mathbb{P}^1_\C \times \mathbb{P}^1_\C$ equal to $V \cup \{(\infty, \infty)\}$, by a result of K\"uhne \cite[Theorem~4]{Kuhne13};
	\item a few very specific classes of curves $V \subset \mathbb{A}^2_\C$ and surfaces $V \subset \mathbb{A}^3_\C$ defined over $\Q$, see \cite{AllombertBiluMadariaga15, BiluLucaMadariaga16, Riffaut19, LucaRiffaut19} and \cite{Fowler20, Fowler23, BiluGunTron26} respectively;
	\item when $V \subset \AN$ is a hypersurface defined by an equation of the form 
	\[\quad \quad \mbox{either } x_i P(x_1, \ldots, x_n) = 1 \mbox{ or } \Psi_N(x_i, x_k) P(x_1, \ldots, x_n) = 1,\] 
	 where $P$ is a polynomial with algebraic integer coefficients and $\Psi_N$ denotes the modular polynomial defined in \cite[\S6.1]{Zagier08}. 
	These results are due to Bilu, Habegger, and K\"uhne \cite[Corollary~1.2]{BiluHabeggerKuhne18} and Li \cite[Corollary~1.6]{Li21} respectively.
\end{enumerate}

Note that the hypersurfaces in (3) in fact contain no special subvarieties, because doing so would contradict the main result of \cite{BiluHabeggerKuhne18} (respectively, of \cite{Li21}) that singular moduli (respectively, modular polynomials $\Psi_N$ evaluated at singular moduli) are never units in the ring of algebraic integers.

Our Theorem~\ref{thm:main} therefore provides the first class of examples of subvarieties of $\AN$ of arbitrary dimension and with non-empty special locus for which a uniform form of Andr\'e--Oort is known effectively. 

All other currently known effective Andr\'e--Oort results for $\AN$ do not possess the same level of uniformity.
Such results include when $V \subset \AN$ is:
\begin{itemize}
	\item  a curve, proved independently by K\"uhne \cite{Kuhne12} and Bilu, Masser, and Zannier \cite{BiluMasserZannier13}, see also variant proofs by W\"ustholz \cite{Wustholz14} and, under some additional hypotheses, Papas \cite[Corollary~1.7]{Papas26};
	\item a linear subvariety, due to Bilu and K\"uhne \cite[Theorem~1.1]{BiluKuhne20};
	\item  a \textit{hereditarily degree non-degenerate hypersurface} (see \cite[Definition~2]{Binyamini19}), by a result of Binyamini \cite[Corollary~4]{Binyamini19}.
\end{itemize}
These results each give a bound on $\max_i \lvert \Delta_i \rvert$ which also depends on some measure of the height of the subvariety $V$.

Note that the set of hypersurfaces covered by Theorem~\ref{thm:main} strictly includes the set of hyperplanes of $\AN$ and is in turn strictly contained in the set of hereditarily degree non-degenerate hypersurfaces of $\AN$.
Theorem~\ref{thm:main} is thus both a uniform generalisation of \cite[Theorem~1.1]{BiluKuhne20} and a uniform version of a special case of  \cite[Corollary~4]{Binyamini19}.
(The distinction between hyperplanes and linear subvarieties of $\AN$ is not significant for applications to Andr\'e--Oort, see \cite[\S2.1]{BiluKuhne20}.)

The dependence of the constant $c(m ,n , d)$ in Theorem~\ref{thm:Pila} on each of the parameters $m, n, d$ is necessary.
If, however, one wants to bound only the number of special points in $V \setminus V^\mathrm{sp}$, rather than the associated discriminants, then the dependence on the degree of the field of definition can be eliminated.
It follows from a result of Scanlon \cite{Scanlon04} that, for $m, n \in \Z_{> 0}$, there exists an ineffective constant $c(m, n)$ with the property that the number of special points in $V \setminus V^\mathrm{sp}$ is at most $c(m, n)$ for $V \subset \AN$ a hypersurface of degree $m$.

\subsection{Structure of the paper}

First, we recall in Section~\ref{sec:prelim} some basic facts which will be useful throughout the paper.
Theorem~\ref{thm:main} is proved in Sections~\ref{sec:nonunif} and \ref{sec:unif}.
The proof of Theorem~\ref{thm:tups} occupies Sections~\ref{sec:tup1} and \ref{sec:tup2}.

	\section{Preliminaries}\label{sec:prelim}
	
	\subsection{Heights}

Let $h \colon \mathbb{A}_{\alg}^n \to \R_{\geq 0}$ denote the absolute logarithmic height,
which is defined as in \cite[\S1.5]{BombieriGubler06}.
For $x \in \alg$, let $\lceil x \rceil$ denote the ``house'' of $x$, i.e.~the maximal absolute value of a $\Q$-conjugate of $x$.

The following simple result bounds the height of the coefficients needed to witness a minimal linear dependence among algebraic integers $x_1, \ldots, x_n$.
Crucially for its application in the proofs of Theorems~\ref{thm:main} and \ref{thm:tups}, this bound is linear in the heights of the $x_i$.

\begin{prop}\label{prop:htcoeff}
	Let $n \in \Z_{>0}$.
	Let $x_1, \ldots, x_n$ be pairwise distinct algebraic integers.
	Let $K$ be a number field.
	Suppose that $x_1, \ldots, x_n$ are linearly dependent over $K$ and minimal for this property.
	Then there exist $a_1, \ldots, a_n \in K^\times$ such that $a_1 = 1$,
	\[  a_1 x_1 +  \ldots + a_n x_n = 0, \mbox{and}\]
	\[ \max_{i } h(a_i) \leq 2  (n-1) \sum_{i=1}^n h(x_i) + (n-1) \log^+ (n-1).\]
\end{prop}

\begin{proof}
	If $n = 1$, then $x_1 = 0$, and hence setting $a_1 = 1$ suffices.
	So assume $n \geq 2$. 
	Since $x_1, \ldots, x_n$ are minimally linearly dependent over $K$, there exist unique elements $b_2, \ldots, b_n \in K^\times$ such that $ x_1 = b_2 x_2 + \ldots + b_n x_n$.
	Since $x_2, \ldots, x_n$ are linearly independent over $K$, there exist $\sigma_1, \ldots, \sigma_{n-1} : K(x_1, \ldots, x_n) \hookrightarrow \C$ embeddings over $K$ such that the matrix
	\[ M_1 = \begin{pmatrix}
		\sigma_{1}(x_2) & \ldots & \sigma_1(x_n)\\
		\vdots & & \vdots\\
		\sigma_{n-1}(x_2) & \ldots & \sigma_{n-1}(x_n)
	\end{pmatrix}
	\]
	is non-singular.
	Observe that
	\[ M_1 (b_2, \ldots, b_n)^t = (\sigma_1(x_1), \ldots, \sigma_{n-1}(x_1))^t.\]
	
	For $i \in \{1, \ldots, n\}$, let $v_i = (\sigma_1(x_i), \ldots, \sigma_{n-1}(x_i))^t$.
	So we may write $M_1 = (v_2, \ldots, v_n)$.
	For $i \in \{2, \ldots, n\}$, let $M_i$ denote the matrix obtained from $M_1$ by replacing the column $v_i$ with $v_1$.
	Cramer's rule then gives that, for every $i \in \{2, \ldots, n\}$,
	\[ b_{i} = \frac{\det M_{i}}{\det M_1}.\]
	Note that $\det M_{i}$ is an algebraic integer for every $i \in \{1, \ldots, n\}$, because $x_1, \ldots, x_n$ are algebraic integers.
	Hence, for every $i \in \{1, \ldots, n\}$,
	\[ h(\det M_i) = \frac{1}{[L : \Q]} \sum_{\tau \colon L \hookrightarrow \C} \log^+ \lvert \tau(\det M_i) \rvert,\]
	where $L$ is a normal number field containing $x_1, \ldots, x_n$.
	
	For every $i \in \{1, \ldots, n\}$ and $\tau \colon L \hookrightarrow \C$, Hadamard's inequality gives that
	\begin{align*}
		\log^+ \lvert \tau(\det M_i) \rvert &= \log^+ \lvert \det (\tau(v_j) : j \in \{1, \ldots, n\} \setminus \{i\}) \rvert\\
		&\leq \log^+ \left( \prod_{\substack{j=1\\ j \neq i}}^{n} \left( \sum_{k=1}^{n-1} \lvert \tau(\sigma_k(x_j)) \rvert^2 \right)^{1/2} \right)\\
		&\leq \frac{n-1}{2} \log^+(n-1) + \sum_{\substack{j=1\\ j \neq i}}^{n} \max_{1 \leq k \leq n-1} \log^+ \lvert  \tau(\sigma_k(x_j)) \rvert.
	\end{align*}
	Therefore, for every $i \in \{2, \ldots, n\}$,
	\begin{align*}
		h(b_i) &\leq h(\det M_i) + h(\det M_1)\\
&\leq (n-1) \log^+(n-1) + \frac{2}{[L : \Q]} \sum_{\tau \colon L \hookrightarrow \C} \sum_{j=1}^{n} \max_{1 \leq k \leq n-1} \log^+ \lvert  \tau(\sigma_k(x_j)) \rvert\\
&\leq (n-1) \log^+(n-1) + \frac{2}{[L : \Q]} \sum_{j=1}^{n}  \sum_{\tau \colon L \hookrightarrow \C}  \sum_{k=1}^{n-1} \log^+ \lvert  \tau(\sigma_k(x_j)) \rvert\\	
&\leq (n-1) \log^+(n-1) + 2(n-1) 	\sum_{j=1}^{n} h(x_j).
	\end{align*}
Now set $a_1 = 1$ and $a_i = -b_i$ for $i \in \{2, \ldots, n\}$.
\end{proof}

	\subsection{Singular moduli}\label{subsec:singmod}
	
	\subsubsection{Basic properties}
	
	A \textit{singular modulus} $x$ is the $j$-invariant of an elliptic curve with complex multiplication.
	The \textit{discriminant} $\Delta$ of $x$ is defined to be the discriminant of the endomorphism order of a corresponding elliptic curve.
	In particular, $\Delta$ is a negative integer and $\Delta \equiv 0, 1 \bmod 4$.
	The \textit{fundamental discriminant} $D$ of $x$ is defined to be the discriminant of the imaginary quadratic field $\Q(\sqrt{\Delta})$.
	One has that $\Delta = f^2 D$ for some $f \in \Z_{>0}$.

Singular moduli are algebraic integers.
The singular moduli of a given discriminant form a complete set of conjugates over $\Q$.
If $x$ is a singular modulus of discriminant $\Delta$, then $\Q(\sqrt{\Delta}, x) / \Q$ is a Galois extension and
\[ [\Q(x) : \Q] = \cl(\Delta),\]
where $\cl(\Delta)$ denotes the class number of the discriminant $\Delta$.
See, for example, \cite[Lemma~9.3 \& Proposition~13.2]{Cox22}.

\subsubsection{Class number bounds}\label{subsub:class}

Famous results of Goldfeld \cite{Goldfeld76} and Gross--Zagier \cite{GrossZagier86} give an effective lower bound for $\cl(\Delta)$ in terms of $\lvert \Delta \rvert$.

\begin{prop}\label{prop:class}
	Let $\Delta < 0$ be such that $\Delta \equiv 0, 1 \bmod 4$. Then
	\[ \cl(\Delta) \geq \frac{1}{1000} (\log \lvert \Delta \rvert)^{1/2}.\]
\end{prop}

\begin{proof}
	Let $\Delta = f^2 D$, where $D$ is the associated fundamental discriminant.
	In particular, $\lvert D \rvert\geq 3$.
	By~e.g.~\cite[Proposition~2.6]{Fowler26},
	\[ \cl(D)^2 \geq \frac{1}{42000} \log \lvert D \rvert.\]
	The formula for $\cl(\Delta)$ in \cite[Theorem~7.24]{Cox22} implies that
	\[ \cl(\Delta)^2 \geq \frac{1}{9} \varphi(f)^2 \cl(D)^2.\]
	And so the elementary bound $\varphi(f) \geq \sqrt{f/2}$ gives that
	\[ \cl(\Delta)^2 \geq \frac{1}{378000} \frac{f}{2} \log \lvert D \rvert.\]
	Since $f/2 \geq 0.4(1+2 \log f)$ for every $f \geq 1$ and also $\log \lvert D \rvert \geq 1$,
	\[\cl(\Delta)^2 \geq \frac{0.4}{378000} (1 + 2 \log f) \log \lvert D \rvert \geq \frac{1}{10^6} (\log \lvert D \rvert + 2 \log f). \qedhere\]
\end{proof}

A theorem of Tatuzawa \cite{Tatuzawa51} implies (see \cite[\S2.6]{BiluKuhne20}) that
\begin{align}\label{eq:Tat}
	\cl(D) > (6.7 \times 10^{-3}) \lvert D \rvert^{5/12}
\end{align}
for every fundamental discriminant $D$, apart from at most one exception.
Denote by $D_*$ the unique fundamental discriminant for which the inequality \eqref{eq:Tat} is false, if such a fundamental discriminant exists.
If no such fundamental discriminant exists, then we will subsequently adopt the convention that the condition $D \neq D_*$ holds for every fundamental discriminant $D$.

A similar inequality to \eqref{eq:Tat} holds for non-fundamental discriminants, thanks to the class number formula in \cite[Theorem~7.24]{Cox22}.
In particular, if $\Delta < 0$ is such that $\Delta \equiv 0, 1 \bmod 4$ and the fundamental discriminant $D$ associated to $\Delta$ satisfies $D \neq D_*$, then, by~\cite[(17)]{BiluKuhne20},
\begin{align}\label{eq:Tatnonfund}
	\cl(\Delta) > (7.4 \times 10^{-4}) \lvert \Delta \rvert^{5/12}.
\end{align}

We will also use the following upper bound for the class number.

\begin{prop}[{\cite[Proposition~2.8]{Fowler26}}]\label{prop:classup}
	Let $\Delta < 0$ be such that $\Delta \equiv 0, 1 \bmod 4$. Then $\cl(\Delta) \leq  \lvert \Delta \rvert^{2/3}$.
	\end{prop}

\subsubsection{Archimedean and height bounds}

	Let $j \colon \h \to \C$ denote the modular $j$-function.
The $j$-function restricts to a bijection $\mathfrak{F} \to \C$, where $\mathfrak{F}$ denotes the standard fundamental domain for the action of $\SL$ on $\h$.	
The map
\begin{align}\label{eq:jbi}
	(a, b, c) \mapsto j\left(\frac{-b + i \lvert \Delta \rvert^{1/2}}{2a}\right)
\end{align}
is therefore a bijection between the set
\begin{align*}
	T_\Delta = \{&(a, b, c) \in \Z^3 : \Delta = b^2 - 4 ac, \gcd(a, b, c) = 1,\\
	&\mbox{ and either } -a < b \leq a < c \mbox{ or } 0 \leq b \leq a = c\}
\end{align*}
and the set of singular moduli of discriminant $\Delta$.
We have the following Archimedean estimate for singular moduli.

\begin{prop}[{\cite[Lemma~1]{BiluMasserZannier13}}]\label{prop:jbd}
	Let $x$ be a singular modulus of discriminant $\Delta$ whose preimage under the map \eqref{eq:jbi} is $(a, b, c) \in T_\Delta$.
	Then
	\[ \left\lvert \left\lvert x \right\rvert - e^{ \pi \lvert \Delta \rvert^{1/2} / a} \right\rvert \leq 2079.\]
\end{prop}

Among the singular moduli of a given discriminant $\Delta$, there is a unique largest (in absolute value) singular modulus, which is called the \textit{dominant} singular modulus of discriminant $\Delta$.
Under the bijection \eqref{eq:jbi}, the dominant singular modulus of discriminant $\Delta$ is the image of the unique triple $(a, b, c) \in T_\Delta$ with $a = 1$.

\begin{prop}[{\cite[Lemmas~5.1 \& 5.3]{BiluLucaMasser17}}]\label{lem:dombig}
	Let $\Delta < 0$ be such that $\Delta \equiv 0, 1 \bmod 4$.
	Let $x$ be the dominant singular modulus of discriminant $\Delta$.
	If $\Delta \neq -3$, then $\lvert x \rvert > 1$.
	Let $y$ be a singular modulus of discriminant $\Delta'$.
	If $x \neq y$ and $\lvert \Delta' \rvert \leq \lvert \Delta \rvert$, then $\lvert y \rvert < \lvert x \rvert$.
\end{prop}

The following height bound for singular moduli will be crucial to our proof of Theorem~\ref{thm:main}.
One may improve the exponent of $\cl(\Delta)$ to $1 - \epsilon$, at the cost of worsening the constant, see e.g.~the non-explicit \cite[Lemma~3]{Kuhne13}.

\begin{prop}\label{prop:htsing}
	Let $x$ be a singular modulus of discriminant $\Delta$. 
	Then
	\[ h(x) \leq 9 \sqrt{2} \frac{\lvert \Delta \rvert^{1/2}}{\cl(\Delta)^{1/2}}.\]
\end{prop}

\begin{proof}
	By the proof of \cite[Lemma~2.10]{Riffaut19}, we have that
	\[ h(x) \leq \frac{9}{2} \frac{\lvert \Delta \rvert^{1/2}}{ \cl(\Delta)} \sum_{(a, b, c) \in T_\Delta} \frac{1}{a}.\]
	Thus, for any constant $M$,
	\begin{align*}
		h(x) &\leq \frac{9}{2} \frac{\lvert \Delta \rvert^{1/2}}{ \cl(\Delta)} \left( \sum_{\substack{(a, b, c) \in T_\Delta\\ a \leq M}} \frac{1}{a} + \sum_{\substack{(a, b, c) \in T_\Delta\\ a > M}} \frac{1}{a} \right).
	\end{align*}
Since $a$ and $b$ together uniquely determine $c$ for a triple $(a, b, c) \in T_\Delta$ and $-a < b \leq a$, there are at most $2a$ triples in $T_\Delta$ with first coordinate equal to a given $a$.
Observe also that $\# T_\Delta  = \cl(\Delta)$.
Consequently, for every $M > 0$,
\begin{align*}
	h(x) &\leq \frac{9}{2} \frac{ \lvert \Delta \rvert^{1/2}}{ \cl(\Delta)} \left( \sum_{a \leq M} 2a\frac{1}{a} + \sum_{i=1}^{\cl(\Delta)} \frac{1}{M} \right)\\
	&\leq \frac{9}{2} \frac{ \lvert \Delta \rvert^{1/2}}{ \cl(\Delta)} \left(2M + \frac{\cl(\Delta)}{M}\right).
\end{align*}
	The stated result then follows by setting $M = \cl(\Delta)^{1/2}/\sqrt{2}$.
\end{proof}

In particular, Propositions~\ref{prop:class} and \ref{prop:htsing} may be combined to show explicitly that the height of a singular modulus of discriminant $\Delta$ is $o(\lvert \Delta \rvert^{1/2})$.
It is this fact which we will rely on to prove Theorems~\ref{thm:main} and \ref{thm:tups}.

\subsection{Special subvarieties}\label{sec:specials}

For $N \in \Z_{> 0}$, let $\Phi_N \in \Z[X, Y]$ denote the classical modular polynomial of level $N$, see e.g.~\cite[\S11B]{Cox22}.
In particular, $\Phi_1(X, Y) = X - Y$ and $\Phi_N(X, Y) = \Phi_N(Y, X)$ for every $N \geq 2$.

For $x, y \in \C$, let $E_x, E_y$ be elliptic curves over $\C$ with $j$-invariants equal to $x, y$ respectively.
Then $\Phi_N(x, y) = 0$ if and only if there exists a cyclic $N$-isogeny $E_x \to E_y$, see \cite[Proposition~14.11]{Cox22}.

A \textit{special subvariety} of $\AN$ is an irreducible component of an algebraic subset of $\AN$ defined by some (possibly empty) collection of equations of the form $\Phi_N(X_i, X_k) = 0$, where $N \in \Z_{> 0}$ and $i, k \in \{1, \ldots, n\}$. 
Note that $i = k$ is allowed in this definition.
One has \cite[\S13A]{Cox22} that
\[ \{ x : x \mbox{ is a singular modulus}\} = \{ x \in \C : \Phi_N(x, x) = 0 \mbox{ for some } N \geq 2\}.\]

\begin{prop}\label{prop:specials}
		Let $m, n \in \Z_{>0}$.
	Let $V \subset \AN$ be a hypersurface defined by an equation
	\[ a_1 x_1^m + \ldots + a_n x_n^m = b,\]
	where $a_1, \ldots, a_n, b \in \alg$.
	Let $k \in \{0, \ldots, n\}$.
	If $S$ is a maximal special subvariety of $V$ with $\dim S = k$, then there exist pairwise disjoint subsets $I_0, \ldots, I_{k} \subset \{1, \ldots, n\}$ and, for each $i \in I_0$, a singular modulus $x_i$ such that:  
	\begin{align*}
		I_0 \cup \ldots &\cup I_k = \{1, \ldots, n\};\\
		&\sum_{i \in I_0} a_i x_i^m = b; \\
	\mbox{if } j \neq 0, &\mbox{ then }	I_j \neq \emptyset \mbox{ and } \sum_{i \in I_j} a_i = 0; \mbox{ and}\\
		 S = \{ (z_1, \ldots, z_n) \in \AN&: \forall i \in I_0 \, \, z_i = x_i \mbox{ and }\forall j \neq 0 \, \, \forall i \in I_j \, \, z_i = z_{i_j}\},
		 \end{align*}
	 where, for each $j \neq 0$, the index $i_j$ is the minimal element of $I_j$.
	 	 Conversely, every set $S \subset \AN$ of this form is a special subvariety of $V$ with $\dim S = k$.
\end{prop}

\begin{proof}
	The hypersurface $V$ is ``hereditarily degree non-degenerate'' in the terminology of \cite[Definition~2]{Binyamini19}.
	Hence, by \cite[Corollary~4]{Binyamini19}, a maximal special subvariety $S \subset V$ is an irreducible component of an algebraic set defined by equations of the form $z_i = z_k$ and $z_i = x$, where $x$ is a singular modulus.
	Such an algebraic set is itself irreducible, so we may write 
	\begin{align*}
		S = \{ (z_1, \ldots, z_n) \in \AN: \forall i \in I_0 \, \, z_i = x_i \mbox{ and }\forall j \neq 0 \, \, \forall i \in I_j \, \, z_i = z_{i_j}\},
	\end{align*}
for some pairwise disjoint subsets $I_0, \ldots, I_k \subset \{1, \ldots, n\}$, which are all non-empty except possibly for $I_0$, and some singular moduli $x_i$ for $i \in I_0$.
Clearly, $\dim S = k$.
Since $S \subset V$, we must have, for every $j \neq 0$, that
\[\sum_{i \in I_j} a_i = 0, \]
and therefore $\sum_{i \in I_0} a_i x_i^m = b$ also.
Conversely, it is clear that any set $S$ of this form is a $k$-dimensional special subvariety of $V$.
\end{proof}

\begin{remark}
	Since every special subvariety of $\AN$ contains a Zariski-dense set of special points \cite[Aside~1.4]{Pila11}, one could also recover Proposition~\ref{prop:specials} directly from our Theorem~\ref{thm:mindep}, rather than appealing to \cite[Corollary~4]{Binyamini19}.
\end{remark}

\section{Non-uniform André--Oort}\label{sec:nonunif}

We will deduce Theorem~\ref{thm:main} from the following result, which is not uniform in the height of the coefficients.
This result may be of some independent interest, because, unlike Theorem~\ref{thm:main}, it is uniform in the exponent $m$.
This uniformity with respect to $m$ appears to be new, even compared to the ineffective results of Pila \cite{Pila11}.

\begin{thm}\label{thm:nonunif}
	Let $n, d \in \Z_{>0}$.
	Let $K$ be a normal number field with $[K : \Q] \leq d$.
	Suppose that $x_1, \ldots, x_n$ are pairwise distinct singular moduli of respective discriminants $\Delta_1, \ldots, \Delta_n$ such that
	\[ a_1 x_1^m + \ldots + a_n x_n^m = b\]
	for some $m \in \Z_{> 0}$ and some $a_1, \ldots, a_n \in K^\times$ and $b \in K$.
	Then
	\begin{align}\label{eq:nonunif}
		\max_{1 \leq i \leq n} \lvert \Delta_i \rvert^{1/2} \leq 2&88 n^2 64^n d^3 h(a_1, \ldots, a_n, b) \nonumber\\ 
		&+ (7.6 \times 10^{10}) (2.1 \times 10^4)^n (n+1)^{4n+6} d^4.
	\end{align}
\end{thm}

Bilu and K\"uhne \cite[Lemma~3.1]{BiluKuhne20} proved the $m = 1$ case of Theorem~\ref{thm:nonunif}.
Our proof of Theorem~\ref{thm:nonunif} will be based on their approach to \cite[Lemma~3.1]{BiluKuhne20}.
The main modifications that we make to their method are necessitated by our desire to obtain uniformity in the exponent $m$.

Binyamini \cite[Corollary~4]{Binyamini19} showed that, in the setting of Theorem~\ref{thm:nonunif}, there exists an effective constant $c(m, n, d, h(a_1, \ldots, a_n, b))$ such that
\[ \max_{1 \leq i \leq n} \lvert \Delta_i \rvert^{1/2} \leq c(m, n, d, h(a_1, \ldots, a_n, b)).\]
He does not though calculate this constant explicitly, or even determine its dependence on $h(a_1, \ldots, a_n, b)$.
For our proof of Theorem~\ref{thm:main}, it is essential that the right-hand side of \eqref{eq:nonunif} is linear in $h(a_1, \ldots, a_n, b)$.

\subsection{Two general reduction steps}\label{sec:red}

First, we adapt two steps from the proof of \cite[Lemma~3.1]{BiluKuhne20} to our more general context.
We will then use these results in the proof of Theorem~\ref{thm:nonunif}.

\subsubsection{Controlling fields via equations}\label{subsec:fields}

Given $a_1, \ldots, a_n \in \alg$, we denote by $\mathrm{Ncl}(a_1, \ldots, a_n)$ the normal closure of the field extension $\Q(a_1, \ldots, a_n) / \Q$. 
The following proposition is the analogue of Step~1 of \cite[\S3]{BiluKuhne20}.

\begin{prop}\label{prop:reduce}
	Let $S \subset \{ k < 0 : k \equiv 0, 1 \bmod 4\}$.
	Suppose that, for $m, n, d \in \Z_{> 0}$ and $l \in \R_{\geq 0}$, there exists a constant $c_S(m, n, d, l)$ with the following property: if $x_1, \ldots, x_n$ are pairwise distinct singular moduli with respective discriminants $\Delta_1, \ldots, \Delta_n \in S$ and such that
	\[ a_1 x_1^m + \ldots + a_n x_n^m = 0\]
	for some $a_1, \ldots, a_n \in \alg^\times$ with
	\[ [\mathrm{Ncl}(a_1, \ldots, a_n) : \Q] \leq d \mbox{ and } h(a_1, \ldots, a_n) \leq l,\]
	then 
	\[ \max_{1 \leq i \leq n} \lvert \Delta_i \rvert^{1/2} \leq c_S(m, n, d, l).\]
	Then the following property holds: for $m, n, d \in \Z_{> 0}$ and $l \in \R_{\geq 0}$, if $x_1, \ldots, x_n$ are pairwise distinct singular moduli  with respective discriminants $\Delta_1, \ldots, \Delta_n \in S$ and
	\[ L = \Q(a_1 x_1^m + \ldots + a_n x_n^m)\]
	for some $a_1, \ldots, a_n \in \alg^\times$ with
	\[ [\mathrm{Ncl}(a_1, \ldots, a_n) : \Q] \leq d \mbox{ and } h(a_1, \ldots, a_n) \leq l,\]
	then, for every $i \in \{1, \ldots, n\}$, either
	\[ \lvert \Delta_i \rvert^{1/2} \leq \max_{1 \leq r \leq 2n} c_S(m, r, d, 2l + \log 2),\]
	or
	\[ \left[L\left(\sqrt{\Delta_i}, x_i\right) : L\left(\sqrt{\Delta_i}\right) \right] \leq \# \{k \in \{1, \ldots, n\} : \Delta_k = \Delta_i\}.\]
\end{prop}

\begin{proof}
	Let $m, n, d \in \Z_{> 0}$, $l \in \R_{\geq 0}$, and $S$ be given. 
	Suppose that a constant $c_S(m, n, d, l)$ satisfying the hypotheses of the proposition is also given. 
	Let $x_1, \ldots, x_n$ be pairwise distinct singular moduli with respective discriminants $\Delta_1, \ldots, \Delta_n \in S$.
	Let $a_1, \ldots, a_n \in \alg^\times$ with
	\[ [\mathrm{Ncl}(a_1, \ldots, a_n) : \Q] \leq d \mbox{ and } h(a_1, \ldots, a_n) \leq l.\]
	Let
	\[ L = \Q(a_1 x_1^m + \ldots + a_n x_n^m).\]
	
	Suppose that $i \in \{1, \ldots, n\}$ is such that
	\[ \left[L\left(\sqrt{\Delta_i}, x_i\right) : L\left(\sqrt{\Delta_i}\right) \right] > \# \{k \in \{1, \ldots, n\} : \Delta_k = \Delta_i\}.\]
	The extension $L(\sqrt{\Delta_i}, x_i) / L(\sqrt{\Delta_i})$ is Galois and is generated by $x_i$. 
	Since
	\[ \# \Gal\left(L\left(\sqrt{\Delta_i}, x_i\right) / L\left(\sqrt{\Delta_i}\right)\right) > \# \{k \in \{1, \ldots, n\} : \Delta_k = \Delta_i\},\]
	there exists some $\sigma \in \Gal(L(\sqrt{\Delta_i}, x_i) / L(\sqrt{\Delta_i}))$ such that
	\[ \sigma(x_i) \notin \{x_1, \ldots, x_n\}.\]
	On the other hand,
	\[ \sigma(a_1 x_1^m + \ldots + a_n x_n^m) = a_1 x_1^m + \ldots + a_n x_n^m. \]
	
	Lift $\sigma$ to an element $\hat{\sigma} \in \Gal(\alg / \Q)$. 
	We obtain that
	\[ a_1 x_1^m + \ldots + a_n x_n^m = \hat{\sigma}(a_1) \hat{\sigma}(x_1)^m + \ldots + \hat{\sigma}(a_n) \hat{\sigma}(x_n)^m.\]
	
	We may rewrite this as an equation
	\begin{align}\label{eq:fieldhom}
		\beta_1 y_1^m + \ldots + \beta_r y_r^m = 0
	\end{align}
	among pairwise distinct singular moduli $y_1, \ldots, y_r$, where $r \in \{1, \ldots, 2n\}$ and $\beta_1, \ldots, \beta_r \in \alg^\times$. 
	Since $\hat{\sigma}(x_i)$ is not among $x_1, \ldots, x_n$, we may assume, by relabelling as necessary, that $y_1 = \hat{\sigma}(x_i)$.
	Note also that the discriminants of $y_1, \ldots, y_r$ are a subset of the set
	\[  \{\Delta_1, \ldots, \Delta_n\} \subset S.\]
	The coefficients $\beta_1, \ldots, \beta_r$ all belong to the set
	\[ \{a_k : 1 \leq k \leq n\} \cup \{-\hat{\sigma}(a_k) : 1 \leq k \leq n\} \cup \{a_k-\hat{\sigma}(a_j) : 1 \leq k, j \leq n\}. \]
	So $h(\beta_1, \ldots, \beta_r) \leq 2l + \log 2$ by \cite[Proposition~1.5.15]{BombieriGubler06} and
	\[ \Q(\beta_1, \ldots, \beta_r) \subset \mathrm{Ncl}(a_1, \ldots, a_n).\]
	Therefore, the equation~\eqref{eq:fieldhom} implies, by hypothesis, that
	\[\lvert \Delta_i \rvert^{1/2}  \leq \max_{1 \leq s \leq 2n} c_S(m, s, d, 2l + \log 2). \qedhere \]
\end{proof}

\subsubsection{Handling distinct fundamental discriminants}

The following result generalises the first part of Step 4 in \cite[\S3]{BiluKuhne20}.

\begin{prop}\label{prop:step4}
	Let $r \geq 2$ and $n_1, \ldots, n_r \in \Z_{> 0}$. 
	Let $D_1, \ldots, D_r$ be pairwise distinct fundamental discriminants. 
	For each $i \in \{1, \ldots, r\}$, let $x_{i, j}$ for $j \in \{1, \ldots, n_i\}$	be pairwise distinct singular moduli of respective discriminant $\Delta_{i, j}$ and fundamental discriminant $D_i$.
	Suppose that
	\[ \sum_{i =1}^r \sum_{j=1}^{n_i} a_{i, j} x_{i, j}^m = b\]
	for some $m \in \Z_{> 0}$ and $a_{i, j} \in \alg^\times$, $b \in \alg$.
	For $i \in \{1, \ldots, r\}$, let $K_i = \Q(\sqrt{D_i})$ and
	\[ L_i = K_i( a_{i, 1} x_{i, 1}^m + \ldots + a_{i, n_i} x_{i, n_i}^m).\]
	If $i \in \{1, \ldots, r\}$ is such that $D_i \neq D_*$ and $j \in \{1, \ldots, n_i\}$ is such that $[L_i(x_{i, j}) : L_i] \leq n_i$, then
	\[ \lvert \Delta_{i, j} \rvert^{1/2} \leq (3.8 \times 10^{10})(2.1 \times 10^4)^n (n+1)^{4n+6} [N : \Q]^4,\]
	where $n = n_1 + \ldots + n_r$ and $N = \mathrm{Ncl}(a_{1, 1}, \ldots, a_{r, n_r}, b)$.
\end{prop}

\begin{proof}
	We argue exactly as below \cite[(39)]{BiluKuhne20}.
	In particular, this relies on class field theory in the form of the main result of \cite{Kuhne21}.
\end{proof}

\subsection{The equal discriminant case of Theorem~\ref{thm:nonunif}}\label{subsec:eq}

We now prove Theorem~\ref{thm:nonunif} in the special case that all the discriminants which occur are either equal or bounded.
The method adapts that of Step~2 in \cite[\S3]{BiluKuhne20}.

\begin{prop}\label{prop:equalnew}
Let $n \in \Z_{> 0}$ and $B \in \R_{\geq 1}$. 
Let $x_1, \ldots, x_n$ be pairwise distinct singular moduli of respective discriminants $\Delta_1, \ldots, \Delta_n$.
Suppose that $\# \{\Delta_i : \lvert \Delta_i \rvert^{1/2} > B\} = 1$.
If $m \in \Z_{> 0}$ and $a_1, \ldots, a_n \in \alg^\times,$ $b \in \alg$ are such that
\[ a_1 x_1^m + \ldots + a_n x_n^m = b,\]
then, denoting $N_0 = \mathrm{Ncl}(a_1, \ldots, a_n)$, we have that
\[  \max_{1 \leq i \leq n} \lvert \Delta_i \rvert^{1/2} \leq \max \{2B, 2 [N_0 : \Q] h(a_1, \ldots, a_n) + \log^+(\lceil b \rceil) + \log(70n)\}.\]
\end{prop}

\begin{proof}

Relabelling as necessary, assume that $\lvert \Delta_1 \rvert$ is maximal. 
We will assume subsequently that $\lvert \Delta_1 \rvert^{1/2} \geq 2 B$.
In particular, $\lvert \Delta_1 \rvert \geq 4$.
Applying a Galois automorphism, we may assume that $x_1$ is the dominant singular modulus of discriminant $\Delta_1$. 
Hence, $\lvert x_1 \rvert > \max \{1, \lvert x_i \rvert \}$ for every $i \geq 2$ by Proposition~\ref{lem:dombig}. 
The equation
\[ a_1 x_1^m + \ldots + a_n x_n^m = b\]
implies that
\begin{align*}
	\lvert x_1 \rvert^m 
	&\leq \sum_{i=2}^n \lvert a_1^{-1} a_i \rvert \lvert x_i \rvert^m + \lvert a_1^{-1} b \rvert,
\end{align*}
and hence
\begin{align}\label{eq:equal1new}
	\lvert x_1 \rvert &\leq \sum_{i=2}^n \lvert a_1^{-1} a_i \rvert \lvert x_i \rvert \left(\frac{\lvert x_i \rvert}{\lvert x_1 \rvert}\right)^{m-1} + \lvert a_1^{-1} b \rvert \left(\frac{1}{\lvert x_1 \rvert}\right)^{m-1} \nonumber\\
	&\leq \sum_{i=2}^n \lvert a_1^{-1} a_i \rvert \lvert x_i \rvert + \lvert a_1^{-1} b \rvert.
\end{align}

If $i > 1$, then either $\lvert \Delta_i \rvert^{1/2} \leq B \leq \lvert \Delta_1 \rvert^{1/2}/2$, or $\Delta_i = \Delta_1$ and $x_i$ is not dominant since $x_i \neq x_1$.
 Proposition~\ref{prop:jbd} and \eqref{eq:equal1new} therefore imply that
\[ e^{\pi \lvert \Delta_1 \rvert^{1/2}} - 2079 \leq  H_0^{[N_0 : \Q]} ( (n-1) H_0^{[N_0 : \Q]} (e^{\pi \lvert \Delta_1 \rvert^{1/2}/2} + 2079) + \lceil b \rceil),\]
in exactly the same way that \cite[(23)]{BiluKuhne20} implies \cite[(24)]{BiluKuhne20}.
Here $H_0 = \exp h(a_1, \ldots, a_n)$.
We therefore obtain, as in \cite[(25)]{BiluKuhne20}, that
\[ \lvert \Delta_1 \rvert^{1/2} \leq 2 [N_0 : \Q] h(a_1, \ldots, a_n) + \log^+(\lceil b \rceil) + \log(70n). \qedhere\]
\end{proof}

Propositions~\ref{prop:reduce} and \ref{prop:equalnew} together immediately imply the following result.

\begin{lem}\label{lem:samediscfield}
Let $m, n, d \in \Z_{>0}$. 
Let	$x_1, \ldots, x_n$ be pairwise distinct singular moduli of the same discriminant $\Delta$.
Let $a_1, \ldots, a_n \in \alg^\times$ with
\[ [\mathrm{Ncl}(a_1, \ldots, a_n) : \Q] \leq d,\]
and set $L = \Q(a_1 x_1^m + \ldots + a_n x_n^m)$.
Then either
\[ \lvert \Delta \rvert^{1/2} \leq 4 d h(a_1, \ldots, a_n) + 2 d \log (2) + \log (140n),\]
or, for every $i \in \{1, \ldots, n\}$,
\[ [L(\sqrt{\Delta}, x_i) : L(\sqrt{\Delta}) ] \leq n.\]
\end{lem}

\subsection{The equal fundamental discriminant case of Theorem~\ref{thm:nonunif}}\label{subsec:eqfund}

Now we will use Proposition~\ref{prop:equalnew} to prove Theorem~\ref{thm:nonunif} in the case that all the discriminants which occur are either of the same fundamental discriminant or bounded.
The method generalises Step~3 of \cite[\S3]{BiluKuhne20}.

\begin{prop}\label{prop:equalfundnew}
Let $n \in \Z_{> 0}$ and $B \in \R_{\geq 1}$. 
Let $x_1, \ldots, x_n$ be pairwise distinct singular moduli of respective discriminants $\Delta_1, \ldots, \Delta_n$ and fundamental discriminants $D_1, \ldots, D_n$.
Suppose that $\# \{D_i :  \lvert \Delta_i \rvert^{1/2} > B\} = 1$.
If $m \in \Z_{> 0}$ and $a_1, \ldots, a_n \in \alg^\times$ and $b \in \alg$ are such that
\[ a_1 x_1^m + \ldots + a_n x_n^m = b,\]
 then, denoting $N_0 = \mathrm{Ncl}(a_1, \ldots, a_n)$,
\begin{align*}
	\max_{1 \leq i \leq n} \lvert \Delta_i \rvert^{1/2} \leq \max \{2B,\, &18 n^2 8^n [ N_0 : \Q]^3 h(a_1, \ldots, a_n)\\ 
	& + \log^+(\lceil b \rceil) + 21 n^3 8^n [N_0 : \Q]^3 \}.
\end{align*}
\end{prop}

\begin{proof}
Relabelling as necessary, we may assume that $\lvert \Delta_1 \rvert$ is maximal and there exists $k \in \{1, \ldots, n\}$ such that
\[ D_1 = \ldots = D_k \mbox{ and } D_i \neq D_1 \mbox{ if } i > k.\]
In particular, $\lvert \Delta_i \rvert^{1/2} \leq B$ if $i > k$.
Hence, if $\Delta_1 = \ldots = \Delta_k$, then  
\[  \max_{1 \leq i \leq n} \lvert \Delta_i \rvert^{1/2} \leq \max \{ 2B, 2 [N_0 : \Q] h(a_1, \ldots, a_n) + \log^+(\lceil b \rceil) + \log(70n)\}\]
by Proposition~\ref{prop:equalnew}.
So we may assume that $\Delta_1, \ldots, \Delta_k$ are not all equal. 

Let $f_1, \ldots, f_k \in \Z_{> 0}$ be such that $\Delta_i = f_i^2 D_1$ for every $i \in \{1, \ldots, k\}$. 
Relabelling again as necessary, we may assume that $f_1 \geq \ldots \geq f_k$. 
Since $\Delta_1, \ldots, \Delta_k$ are not all equal, there exists some minimal $l \in \{2, \ldots, k\}$ such that $f_1 > f_l$. 
Applying a Galois conjugation as necessary, we may assume that $x_1$ is dominant. 

Assume also that $\lvert \Delta_1 \rvert^{1/2} \geq 2B$, so $\lvert \Delta_i \rvert^{1/2} \leq \lvert \Delta_1 \rvert^{1/2}/2$ for every $i > k$.
Then $\lvert x_1 \rvert \geq 1$, and $\lvert x_1 \rvert > \lvert x_i \rvert$ for every $i \geq 2$ by Proposition~\ref{lem:dombig}. 
Hence, from the equation
\begin{align}\label{eq:10} 
	a_1 x_1^m + \ldots + a_n x_n^m = b,
\end{align}
we obtain, exactly as in the proof of Proposition~\ref{prop:equalnew}, that
\begin{align*}
	\lvert x_1 \rvert &\leq \sum_{i=2}^n \lvert a_1^{-1} a_i \rvert \lvert x_i \rvert + \lvert a_1^{-1} b \rvert.
\end{align*}
By Proposition~\ref{prop:jbd}, we deduce that
\begin{align}\label{eq:added}
	&e^{\pi \lvert \Delta_1 \rvert^{1/2}} - 2079 \nonumber\\ 
	\leq &(n-1)  H_0^{2 [ N_0 : \Q]} \left( e^{\pi \lvert \Delta_1 \rvert^{1/2}} \max \{e^{- \pi \lvert \Delta_1 \rvert^{1/2}/2}, e^{\pi (f_l - f_1) \lvert D_1 \rvert^{1/2}} \} + 2079\right) \nonumber\\ 
	&+ H_0^{[N_0 : \Q]} \lceil b \rceil,
\end{align}
where $H_0 = \exp h(a_1, \ldots, a_n)$.
This inequality coincides with \cite[(28)]{BiluKuhne20}.

The proof now proceeds as in Step~3 of \cite[\S3]{BiluKuhne20}.
Suppose first that
\begin{align}\label{eq:assume}
	\pi(f_1 - f_l) \lvert D_1 \rvert^{1/2} \geq 2 [N_0 : \Q] h(a_1, \ldots, a_n) + \log(2n).
\end{align}
From \eqref{eq:added}, we then obtain, as in \cite[(30)]{BiluKuhne20}, that
\[ \lvert \Delta_1 \rvert^{1/2} \leq 2 [ N_0 : \Q] \log( H_0) + \log^+(\lceil b \rceil) + \log(8400n). \]

Next, suppose that
\begin{align*}
	\pi(f_1 - f_l) \lvert D_1 \rvert^{1/2} < 2 [N_0 : \Q] h(a_1, \ldots, a_n) + \log(2n).
\end{align*}
This implies, as in \cite[(31)]{BiluKuhne20}, that
\[ \frac{\mathrm{lcm}(f_1, f_l)}{f_l} > \frac{\pi  \lvert \Delta_1 \rvert^{1/2}}{2 [N_0 : \Q] h(a_1, \ldots, a_n) + \log(2n)}.\]
Let $K_1 = \Q(\sqrt{D_1})$. Let
\[ G= \Gal(N_0 \cdot K_1(x_1) \cdot K_1(x_l) / N_0 \cdot K_1(x_l)).\]
Then, by the same argument as around \cite[(35)]{BiluKuhne20} and appealing to Lemma~\ref{lem:samediscfield}, either there exists an element $\sigma \in G$ which acts non-trivially on 
\[ a_1 x_1^m + \ldots + a_{l-1} x_{l-1}^m,\]
or else
\begin{align*} 
	\lvert \Delta_1 \rvert^{1/2} &\leq \frac{1}{\pi} 216n^2 [N_0 : \Q]^2 (2 [N_0 : \Q] h(a_1, \ldots, a_n) + \log(2n))\\
	&\leq 138 n^2 [N_0 : \Q]^3  h(a_1, \ldots, a_n) + 69 n^2 \log(2n)[N_0 : \Q]^2,
\end{align*}
in which case we are done.

We may thus assume subsequently that an element $\sigma \in G$ which acts non-trivially on 
\[ a_1 x_1^m + \ldots + a_{l-1} x_{l-1}^m\]
exists. 
Let $\hat{\sigma} \in \Gal(\alg / N_0 \cdot K_1(x_l))$ be a corresponding lift of $\sigma$.
Apply $\hat{\sigma}$ to \eqref{eq:10} and subtract the result from \eqref{eq:10}. 
One thereby obtains an equation
\begin{align*}
	\beta_1 y_1^m + \ldots + \beta_s y_s^m = \gamma,
\end{align*}
where $y_1, \ldots, y_s$ are pairwise distinct singular moduli and $s \leq 2n$. 
Let $S$ denote the set of discriminants of $y_1, \ldots, y_s$.
Then
\[ \Delta_1 \in S \subset \{ \Delta_1, \ldots, \Delta_n\} \setminus \{\Delta_l\},\]
since $\hat{\sigma}$ acts non-trivially on $a_1 x_1^m + \ldots + a_{l-1} x_{l-1}^m$ and trivially on $K_1(x_l)$. 
In particular, 
\[ \# \{\Delta_i \in S : D_i = D_1 \} < \# \{f_1, \ldots, f_k \} \leq k \leq n.\]
Further,  $\beta_1, \ldots, \beta_s \in \mathrm{Ncl}(a_1, \ldots, a_n)$,
\[ h(\beta_1, \ldots, \beta_s) \leq 2 h(a_1, \ldots, a_n) + \log 2\]
by \cite[Proposition~1.5.15]{BombieriGubler06}, and $\lceil \gamma \rceil \leq 2 \lceil b \rceil$.

Now repeat this process. 
At each stage, we either obtain an equation that satisfies the analogue of \eqref{eq:assume}, in which case we may bound $\lvert \Delta_1 \rvert$, or we reduce by (at least) one the number of distinct discriminants $\Delta_i$ with fundamental discriminant equal to $D_1$ that occur in our equation. 
After at most $n - 1$ iterations, we will therefore either have at some stage bounded $\lvert \Delta_1 \rvert$, or we will have obtained an equation in which the only discriminant with fundamental discriminant equal to $D_1$ which occurs is $\Delta_1$. 
In the latter case, Proposition~\ref{prop:equalnew} then implies a bound on $\lvert \Delta_1 \rvert^{1/2}$. 
In either case, we obtain eventually, as in \cite[(37)]{BiluKuhne20},  that
\[ \lvert \Delta_1 \rvert^{1/2} \leq 18 n^2 8^n [ N_0 : \Q]^3 \log(H_0) + \log^+(\lceil b \rceil) + 21 n^3 8^n [N_0 : \Q]^3. \qedhere \]
\end{proof}

Propositions~\ref{prop:reduce} and \ref{prop:equalfundnew} together imply the following result.

\begin{lem}\label{lem:eqfundfield}
Let $m, n, d \in \Z_{>0}$. 
Let	$x_1, \ldots, x_n$ be pairwise distinct singular moduli of respective discriminants $\Delta_1, \ldots, \Delta_n$ which are all of the same fundamental discriminant.
Let $a_1, \ldots, a_n \in \alg^\times$ with
\[ [\mathrm{Ncl}(a_1, \ldots, a_n) : \Q] \leq d,\]
and set $L = \Q(a_1 x_1^m + \ldots + a_n x_n^m)$.
Then, for every $i \in \{1, \ldots, n\}$, either
\[ \lvert \Delta_i \rvert^{1/2} \leq 144 n^2 64^n d^3 h(a_1, \ldots, a_n) + 218 n^3 64^n d^3,\]
or
\[ [L(\sqrt{\Delta_i}, x_i) : L(\sqrt{\Delta_i}) ] \leq n.\]

\end{lem}

\subsection{The general case of Theorem~\ref{thm:nonunif}}

\begin{proof}[Proof of Theorem~\ref{thm:nonunif}]

Let $x_1, \ldots, x_n$ be pairwise distinct singular moduli with respective discriminants $\Delta_1, \ldots, \Delta_n$ and fundamental discriminants $D_1, \ldots, D_n$. 
Suppose that 
\[ a_1 x_1^m + \ldots + a_n x_n^m = b\]
for some $a_1, \ldots, a_n \in \alg^\times$, $b \in \alg$, and $m \in \Z_{> 0}$. 
Let $N = \mathrm{Ncl}(a_1, \ldots, a_n, b)$ and $N_0 = \mathrm{Ncl}(a_1, \ldots, a_n)$.
Set $d = [N : \Q]$ and $d_0 = [N_0 : \Q]$. 

Let
\[ r = \# \{D_1, \ldots, D_n\}.\]
If $r=1$, then Proposition~\ref{prop:equalfundnew} implies that
\begin{align}\label{eq:1fund}
	\max_{1 \leq i \leq n} \lvert \Delta_i \rvert^{1/2} \leq 18 n^2 8^n d_0^3 h(a_1, \ldots, a_n) + \log^+(\lceil b \rceil) + 21 n^3 8^n d_0^3.
\end{align}

So we may assume that $r \geq 2$. 
First, we will bound $\lvert \Delta_i \rvert$ for all those $i$ such that $D_i \neq D_*$.
Recall from Section~\ref{subsub:class} that $D_*$ is the single possible exceptional fundamental discriminant arising in the Tatuzawa bound \eqref{eq:Tat}.
Let $i \in \{1, \ldots, n\}$ be such that $D_i \neq D_*$, and set 
\[ J_i = \{1 \leq k \leq n : D_k = D_i\}.\]
Let 
\[ L_i = \Q\left(\sum_{k \in J_i} a_k x_k^m\right).\]
By Lemma~\ref{lem:eqfundfield}, either
\[ \lvert \Delta_i \rvert^{1/2} \leq 144 n^2 64^n d_0^3 h(a_1, \ldots, a_n) + 218 n^3 64^n d_0^3,\]
or
\[ [L_i(\sqrt{\Delta_i}, x_i) : L_i(\sqrt{\Delta_i})] \leq \# J_i.\]
In the latter case, Proposition~\ref{prop:step4} implies that
\[ \lvert \Delta_i \rvert^{1/2} \leq (3.8 \times 10^{10}) (2.1 \times 10^4)^n (n+1)^{4n + 6} d^4.\]
Thus, in either case,
\begin{align}\label{eq:goodbd}
	\lvert \Delta_i \rvert^{1/2} \leq &144 n^2 64^n d_0^3 h(a_1, \ldots, a_n) \nonumber\\ 
	&+ (3.8 \times 10^{10}) (2.1 \times 10^4)^n (n+1)^{4n + 6} d^4. \end{align}

So if 
\[ D_* \notin \{D_1, \ldots, D_n\},\]
then we are done.
Relabelling as necessary, we may thus assume there exists some $k \in \{1, \ldots, n\}$ such that
\[ D_i = D_* \iff i \in \{1, \ldots, k\}.\]
Thus, by Proposition~\ref{prop:equalfundnew} and \eqref{eq:goodbd}, if $ i \in \{1, \ldots, k\}$, then
\begin{align*} 
	\lvert \Delta_i \rvert^{1/2} \leq \max \{ &2 \big(144 n^2 64^n d_0^3 h(a_1, \ldots, a_n) \\ 
	&+ (3.8 \times 10^{10}) (2.1 \times 10^4)^n (n+1)^{4n + 6} d^4\big),\\
	&144 n^2 64^n d_0^3 h(a_1, \ldots, a_n) + \log^+(\lceil b \rceil) + 21 n^3 8^n d_0^3\}.
\end{align*}

Taking the maximum of this, \eqref{eq:1fund}, and \eqref{eq:goodbd}, we obtain that
\begin{align*}
	\max_{1 \leq i \leq n} \lvert \Delta_i \rvert^{1/2} \leq &288n^2 64^n d^3 h(a_1, \ldots, a_n, b)\\ 
	&+ (7.6 \times 10^{10}) (2.1 \times 10^4)^n (n+1)^{4n + 6} d^4. \qedhere
\end{align*}
\end{proof}

\begin{remark}
To bound those $\lvert \Delta_i \rvert$ with $D_i = D_*$, we used the fact that Propositions~\ref{prop:equalnew} and \ref{prop:equalfundnew} allow some arbitrary discriminants to occur, provided they are all bounded.
This allows us to maintain the uniformity in $m$.
In Bilu and K\"uhne's proof \cite[(40)]{BiluKuhne20}, they bound such $\lvert \Delta_i \rvert$ by absorbing all those terms $a_k x_k^m$ with $D_k \neq D_*$ into the constant term. 
In our setting, following their approach would give a final bound that also depended on $m$.
\end{remark}
	
	\section{Uniform André--Oort}\label{sec:unif}

\subsection{Minimal linear dependencies}

We will use Theorem~\ref{thm:nonunif} to prove the following theorem, which is our first uniform result.
The proof relies on the height bounds in Propositions~\ref{prop:htcoeff} and \ref{prop:htsing}.

	\begin{thm}\label{thm:mindep}
		Let $m, n, d \in \Z_{>0}$.
		Let $K$ be a normal number field with $[K : \Q] \leq d$.
		Let $x_1, \ldots, x_n$ be pairwise distinct singular moduli of respective discriminants $\Delta_1, \ldots, \Delta_n$.
		Suppose that $x_1^m, \ldots, x_n^m$ are linearly dependent over $K$ and minimal for this property.
		Then
		\[ \max_{1 \leq i \leq n} \lvert \Delta_i \rvert^{1/2} \leq \exp(10^{20} 10^{13n} d^{12} m^4).\]
	\end{thm}
	
	\begin{proof}
		If $n = 1$, then $x_1 = 0$, and so $\lvert \Delta_1 \rvert = 3$.
		So we may assume that $n \geq 2$.
		Singular moduli (and hence their $m$th powers) are algebraic integers.
		So, by Proposition~\ref{prop:htcoeff}, there exist $a_1, \ldots, a_n \in K^\times$ such that: $a_1 = 1$,
		\[ a_1 x_1^m +  \ldots + a_n x_n^m = 0,\]
		and 
		\[ h(a_1, \ldots, a_n) \leq (n-1) \left( 2 (n-1) \sum_{i=1}^n h(x_i^m) + (n-1) \log^+ (n-1) \right).\]
		Theorem~\ref{thm:nonunif} implies that
		\[ \max_{1 \leq i \leq n} \lvert \Delta_i \rvert^{1/2} \leq c_1(n, d) h(a_1, \ldots, a_n) + c_2(n, d),\]
		where $c_1(n, d), c_2(n, d)$ are the appropriate constant terms from \eqref{eq:nonunif}.
		So
		\[ \max_{1 \leq i \leq n} \lvert \Delta_i \rvert^{1/2} \leq c_1( n, d) (n-1)^2 \left(2m \sum_{i=1}^n h(x_i) +  \log^+ (n-1)\right) + c_2( n, d).\]
		
		By Proposition~\ref{prop:htsing}, 
		\[ h(x_i) \leq 9 \sqrt{2} \frac{\lvert \Delta_i \rvert^{1/2}}{\cl(\Delta_i)^{1/2}}.\]
		Thus
		\begin{align}\label{eq:bd1}
			\max_{1 \leq i \leq n} \lvert \Delta_i \rvert^{1/2} \leq c_3( n, d) + c_4(m , n, d) \sum_{i=1}^n  \frac{\lvert \Delta_i \rvert^{1/2}}{\cl(\Delta_i)^{1/2}},
			\end{align}
		where
		\begin{align*}
			 &c_3( n, d) = c_1(n, d) (n-1)^2 \log^+ (n-1) + c_2(n, d), \mbox{ and}\\
		&c_4(m, n, d) = 18 \sqrt{2} c_1(n, d) (n-1)^2m.	 
		 \end{align*}
	 One may easily check that
	 \[ c_3(n, d) \leq (100(n+1))^{4n+6} d^4,\]
	 \[ c_4(m, n, d) \leq 10^{3(n+1)} d^3 m.\]
		
		Note that if $\cl(\Delta_i)^{1/2} \leq 2 n c_4(m, n, d)$,
		then Proposition~\ref{prop:class} implies that
			\[ \lvert \Delta_i \rvert^{1/2} \leq \exp((8 \times 10^6) n^4 c_4(m, n, d)^4).\]
		We therefore obtain from \eqref{eq:bd1} that
		\begin{align*}
	\max_{1 \leq i \leq n} \lvert \Delta_i \rvert^{1/2} \leq &c_3( n, d) 
	+ c_4(m , n, d) \sum_{i=1}^n \Big(\frac{\max_{1 \leq i \leq n} \lvert \Delta_i \rvert^{1/2}}{2 n c_4(m, n, d)}\\
	&+\exp((8 \times 10^6) n^4 c_4(m, n, d)^4)\Big).
\end{align*}
		So 
		\[ \max_{1 \leq i \leq n} \lvert \Delta_i \rvert^{1/2} \leq 2 \left( c_3( n, d) + n c_4(m, n, d) \exp((8 \times 10^6) n^4 c_4(m, n, d)^4)  \right). \]
		Simplifying the constant, we obtain that
			\[ \max_{1 \leq i \leq n} \lvert \Delta_i \rvert^{1/2}  \leq \exp(10^{20} 10^{13n} d^{12} m^4). \qedhere \]
	\end{proof}

\begin{remark}\label{rmk:const}
	The constant in Theorem~\ref{thm:mindep} has the form
	\[ \exp(c(n) [K: \Q]^{12} (\deg V)^4).\]
	 This is of comparable quality to the constant $\exp( \alpha(g) ( [K :\Q] \deg V)^{\beta(g)})$ obtained by Binyamini, Jones, Schmidt, and Thomas in their uniform effective Andr\'e--Oort result for fibre powers $\mathcal{E}^{(g)}$ of the Legendre family of elliptic curves \cite[Theorems~1.3 \& 5.10]{BinyaminiSchmidtJonesThomas26}.
	In both cases, the exponential dependence arises from applying the class number bound of Goldfeld and Gross--Zagier (Proposition~\ref{prop:class}), which is logarithmic in $\lvert \Delta \rvert$.
\end{remark}

\subsection{Uniform bounds on maximal special subvarieties}

We will now use Theorem~\ref{thm:mindep} to prove the first of the main results of the paper.

	\begin{proof}[Proof of Theorem~\ref{thm:main}]
			Let $m, n, d \in \Z_{> 0}$. 
			Let $K$ be a number field with $[K : \Q] \leq d$.
		Let $V \subset \AN$ be a hypersurface defined by an equation
		\[ a_1 z_1^m + \ldots + a_n z_n^m = b\]
		for some $a_1, \ldots, a_n, b \in K$. 
		Let $V^\mathrm{sp}$ denote the union of the positive-dimensional special subvarieties of $V$.
		Suppose that $x_1, \ldots, x_n$ are singular moduli of respective discriminants $\Delta_i$ such that
		\[ (x_1, \ldots, x_n) \in V \setminus V^\mathrm{sp}.\]
		
		Let
		\[ I = \{i \in \{1, \ldots, n\} : \forall j < i \, \, x_j \neq x_i\}.\]
		Observe that $\{\Delta_i : i \in I\} = \{\Delta_1, \ldots, \Delta_n\}$.
		For each $i \in I$, let
		\[ J_i = \{ j \in \{1, \ldots, n\} : x_j = x_i\}\]
		and let
		\[ a_i' = \sum_{j \in J_i} a_j.\]
		Note that $a_i' \in K$.
		If $a_i' = 0$ for some $i \in I$, then, by Proposition~\ref{prop:specials},
		\[ \{(z_1, \ldots, z_n) \in \AN : \forall j \in J_i \, \, z_j = z_i\mbox{ and } \forall j \notin J_i \, \, z_j = x_j\}\]
		is a positive-dimensional special subvariety of $V$ that contains $(x_1, \ldots, x_n)$, which is a contradiction.
		So $a_i' \neq 0$ for every $i \in I$.

Note that
		\begin{align}\label{eq:reduced}
			 \sum_{i \in I} a_i' x_i^m = b.
			 \end{align}
		Suppose first that $b=0$.
			Then, for each $i \in I$, there is a subset $I_{i} \subset I$ such that $i \in I_i$ and the set
		\[ \{x_j^m : j \in I_{i}\}\]
		is linearly dependent over the normal closure of $K / \Q$ and minimal for this property.
		We may then apply Theorem~\ref{thm:mindep} to obtain that
		\[ \lvert \Delta_i \rvert^{1/2} \leq \exp(10^{20} 10^{13n} (d!)^{12} m^4),\]
		since the normal closure of $K/\Q$ has degree at most $d!$ over $\Q$.

		We will now show that we may always reduce to the case where $b = 0$.
		Suppose that $b \neq 0$ in \eqref{eq:reduced}.
Recall that $1728$ is the unique singular modulus of discriminant $\Delta = -4$.	
If $x_{i_0} = 1728$ for some $i_0 \in I$, then $\lvert \Delta_{i_0} \rvert^{1/2} = 2$ and we may rewrite \eqref{eq:reduced} as
\[ \sum_{i \in I \setminus \{i_0\}} a_i' x_i^m + \left(a_{i_0}' - \frac{b}{1728^m}\right) x_{i_0}^m = 0, \]
whenceforth the argument from the $b = 0$ case above shows that
\[ \lvert \Delta_i \rvert^{1/2} \leq  \exp(10^{20} 10^{13n} (d!)^{12} m^4)\]
for every $i \in I \setminus \{i_0\}$.

If
\[ 1728 \notin \{x_i : i \in I\},\]
then let $x_0 = 1728$ and $a_0' = -b/1728^m \in K^\times$, so that \eqref{eq:reduced}  implies that
\[  \sum_{i \in I \cup \{0\}} a_i' x_i^m  = 0.\]	
	Hence, in this case, the argument from the $b=0$ case now implies that
	\[ \max_{1 \leq i \leq n} \lvert \Delta_i \rvert^{1/2} \leq \exp(10^{20} 10^{13(n+1)} (d!)^{12} m^4). \qedhere\]
	\end{proof}

Uniform explicit bounds also hold for the positive-dimensional maximal special subvarieties.

\begin{thm}\label{thm:positive}
	Let $m, n, d \in \Z_{>0}$.
	Let $k \in \{1, \ldots, n\}$.
	Let $V \subset \AN$ be a hypersurface defined by an equation
	\[ a_1 x_1^m + \ldots + a_n x_n^m = b,\]
	where $a_1, \ldots, a_n, b \in \alg$ are such that
	\[ [\Q(a_1, \ldots, a_n, b) : \Q] \leq d.\]
	Suppose that $S$ is a maximal special subvariety of $V$ with $\dim S = k$. 
	Then there exist pairwise disjoint subsets $I_0, \ldots, I_{k} \subset \{1, \ldots, n\}$ and, for each $i \in I_0$, a singular modulus $x_i$ of discriminant $\Delta_i$ with the properties that: 
	\begin{align*}
		I_0 \cup \ldots &\cup I_k = \{1, \ldots, n\};\\
			\sum_{i \in I_0} a_i x_i^m = b \mbox{ and } \max_{i \in I_0} &\lvert \Delta_i \rvert^{1/2} \leq \exp(10^{20} 10^{13n} (d!)^{12} m^4);\\
	\mbox{ if } j \neq 0, \mbox{ then }	&I_j \neq \emptyset \mbox{ and } \sum_{i \in I_j} a_i = 0; \mbox{ and}\\	
		S = \{ (z_1, \ldots, z_n) \in \AN&: \forall i \in I_0 \, \, z_i = x_i \mbox{ and }\forall j \neq 0 \, \, \forall i \in I_j \, \, z_i = z_{i_j}\},
	\end{align*}
	where, for each $j \neq 0$, the index $i_j$ is the minimal element of $I_j$.
\end{thm}

\begin{proof}
	By Proposition~\ref{prop:specials}, 
	there exist pairwise disjoint subsets $I_0, \ldots, I_{k} \subset \{1, \ldots, n\}$ 
	and, for each $i \in I_0$, a singular modulus $x_i$ of discriminant $\Delta_i$ which satisfy all the desired properties except possibly the bound on the discriminants
	$\lvert \Delta_i \rvert$ for $i \in I_0$.
	Let
	\[ I_0' = \{i \in I_0: \forall j \in I_0 \mbox{ if } j < i,\mbox{ then } x_j \neq x_i\}.\]
	In particular, the $x_i$ for $i \in I_0'$ are pairwise distinct.
	For each $i \in I_0'$, let 
	\[ J_i = \{ j \in I_0 : x_j = x_i\}.\]
	Then
	\[ \sum_{i \in I_0'} \left(\sum_{j \in J_i} a_j \right) x_i^m = b.\]
	If $i \in I_0'$ is such that
	\[ \sum_{j \in J_i} a_j = 0,\]
	then $S$ is contained in the strictly larger special subvariety of $V$ that arises, in the obvious way, from the partition $I_0 \setminus J_i, I_1, \ldots, I_k, J_i$.
	This contradicts the fact that $S$ is a maximal special subvariety of $V$.
	Hence,
	\[ \sum_{j \in J_i} a_j \neq 0\]
	for every $i \in I_0'$.
	We may therefore argue as below \eqref{eq:reduced} in the proof of Theorem~\ref{thm:main}, to obtain that
	\[ \max_{i \in I_0} \lvert \Delta_i \rvert^{1/2} = \max_{i \in I_0'} \lvert \Delta_i \rvert^{1/2} \leq \exp(10^{20} 10^{13n} (d!)^{12} m^4),\]
	where we use the fact that $\# I_0' \leq \# I_0 \leq n-1$ since $k \geq 1$.
\end{proof}

\subsection{A modular analogue of Mann's theorem}

From Theorem~\ref{thm:main}, we obtain the following corollary, which corresponds to the $l = 0$ case of Theorem~\ref{thm:tups}.
In the special case that $a_1, \ldots, a_n, b \in \Q$, the author \cite[Theorem~1.2]{Fowler26} previously proved a version of this corollary under an additional assumption on the discriminants of the singular moduli involved.

\begin{cor}\label{cor:Mann}
	Let $m, d, n \in \Z_{>0}$.
	Let $x_1, \ldots, x_n$ be singular moduli of respective discriminants $\Delta_1, \ldots, \Delta_n$.
	If $(x_1, \ldots, x_n)$ is a non-degenerate $(m, d, n, 0)$-tuple, then
	\[ \max_{1 \leq i \leq n} \lvert \Delta_i \rvert^{1/2} \leq \exp(10^{20} 10^{13(n+1)} (d!)^{12} m^4).\]
\end{cor}

\begin{proof}[Proof of Corollary~\ref{cor:Mann}]
	
	Let $a_1, \ldots, a_n, b \in \alg$ be coefficients that witness $(x_1, \ldots, x_n)$ being a non-degenerate $(m, d, n, 0)$-tuple.
	In particular,
	\[ a_1 x_1^m + \ldots + a_n x_n^m = b\]
	and $[K : \Q] \leq d$, where $K = \Q(a_1, \ldots, a_n, b)$.

As in the proof of Theorem~\ref{thm:main}, let
\[ I = \{i \in \{1, \ldots, n\} : \forall j < i \, \, x_j \neq x_i\}.\]
Observe that $\{\Delta_i : i \in I\} = \{\Delta_1, \ldots, \Delta_n\}$.
For each $i \in I$, let 
\[J_i = \{ j \in \{1, \ldots, n\} : x_j = x_i\},\]
and let
\[ a_i' = \sum_{j \in J_i} a_j \in K.\]
Since $(x_1, \ldots, x_n)$ is a non-degenerate $(m, d, n, 0)$-tuple, we have that $a_i' \neq 0$ for every $i \in I$.
So, arguing as in the proof of Theorem~\ref{thm:main}, we obtain that
	\[ \max_{1 \leq i \leq n} \lvert \Delta_i \rvert^{1/2} \leq \exp(10^{33} 10^{13n} (d!)^{12} m^4). \qedhere\]
\end{proof}

\begin{remark}
	For $n \leq 3$, the complete lists of $n$-tuples $(x_1, \ldots, x_n)$ of pairwise distinct singular moduli such that $a_1 x_1 + \ldots + a_n x_n \in \Q$ for some $a_1, \ldots, a_n \in \Q^\times$ have been found explicitly \cite{AllombertBiluMadariaga15, Fowler26}.
	Corollary~\ref{cor:Mann} shows that, for each $n \geq 1$, any such $n$-tuple must satisfy
	\[ \max_{1 \leq i \leq n} \lvert \Delta_i \rvert^{1/2} \leq \exp(10^{33} 10^{13n} ).\]
	Unfortunately this bound is too large for it to be computationally feasible to find all such $n$-tuples for a given $n$. 
	For example, when $n = 4$, we obtain the astronomical bound $\max_{1 \leq i \leq 4} \lvert \Delta_i \rvert^{1/2} \leq \exp( 10^{85})$.
\end{remark}

\section{Roots of unity and singular moduli of the same fundamental discriminant}\label{sec:tup1}

\subsection{A non-uniform bound on the singular moduli}\label{subsec:nonunif}

The first step in the proof of Theorem~\ref{thm:tups} is to prove a non-uniform bound on the singular moduli, in the case that they are all of the same fundamental discriminant.

\begin{prop}\label{prop:equalfundnonunifMann}
	Let $d, n \in \Z_{> 0}$ and $l \in \Z_{\geq 0}$. 
	Let $L$ be a normal number field with $[L : \Q] \leq d$.
	Let $x_1, \ldots, x_n$ be pairwise distinct singular moduli of respective discriminants $\Delta_1, \ldots, \Delta_n$, but all of the same fundamental discriminant $D$.
	Let $\zeta_1, \ldots, \zeta_l$ be roots of unity and $m \in \Z_{> 0}$.
	If $a_1, \ldots, a_n \in L^\times$ and $b_1, \ldots, b_l, c \in L$ are such that
	\[ a_1 x_1^m + \ldots + a_n x_n^m + b_1 \zeta_1 + \ldots + b_l \zeta_l = c,\]
	then
	\[ \max_{1 \leq i \leq n}\lvert \Delta_i \rvert^{1/2} \leq 138  n^3 8^n d^3 ( h(a_1, \ldots, a_n, b_1, \ldots, b_l, c) + \log(n + l + 1) + 1).\]
\end{prop}

\begin{proof}
	We may assume that $\lvert \Delta_1 \rvert$ is maximal and $\lvert \Delta_1 \rvert \geq 4$. 
	By applying a suitable automorphism, we may assume further that $x_1$ is the dominant singular modulus of discriminant $\Delta_1$. 
	The equation
	\begin{align}\label{eq:MannEq}
		 a_1 x_1^m + \ldots + a_n x_n^m + b_1 \zeta_1 + \ldots + b_l \zeta_l = c
		\end{align}
	implies that
	\begin{align*}
		\lvert x_1 \rvert^m 
		&\leq \sum_{i=2}^n \lvert a_1^{-1} a_i \rvert \lvert x_i \rvert^m + \sum_{i=1}^l \lvert a_1^{-1} b_i \rvert  + \lvert a_1^{-1} c \rvert.
	\end{align*}
	Proposition~\ref{lem:dombig} implies that $\lvert x_1 \rvert > \max \{1, \lvert x_i \rvert \}$ for every $i \geq 2$. 
Therefore,
	\begin{align}\label{eq:equal2new}
		\lvert x_1 \rvert &\leq \sum_{i=2}^n \lvert a_1^{-1} a_i \rvert \lvert x_i \rvert \left(\frac{\lvert x_i \rvert}{\lvert x_1 \rvert}\right)^{m-1}  + \left( \frac{1}{\lvert x_1 \rvert} \right)^{m-1} \left( \sum_{i=1}^l \lvert a_1^{-1} b_i \rvert  + \lvert a_1^{-1} c \rvert \right) \nonumber\\
		&\leq \sum_{i=2}^n \lvert a_1^{-1} a_i \rvert \lvert x_i \rvert + \sum_{i=1}^l \lvert a_1^{-1} b_i \rvert  + \lvert a_1^{-1} c \rvert .
	\end{align}
Denote $H = \exp( h(a_1, \ldots, a_n, b_1, \ldots, b_l, c))$.

	Suppose first that $\Delta_1 = \ldots = \Delta_n$.
	Then $x_i$ is not dominant for every $i >1$.
	Therefore, Proposition~\ref{prop:jbd} and \eqref{eq:equal2new} imply that
	\begin{align}\label{ineq:samedisczeta}
		e^{\pi \lvert \Delta_1 \rvert^{1/2}} - 2079 &\leq  H^{[L : \Q]} ( (n-1) H^{[L : \Q]} (e^{\pi \lvert \Delta_1 \rvert^{1/2}/2} + 2079) + (l+1) H^{[L: \Q]}) \nonumber\\
		&\leq (n+l) H^{2[L : \Q]} (e^{\pi \lvert \Delta_1 \rvert^{1/2}/2} + 2079).		 
	\end{align}
	It is easy to check that this inequality implies that
	\begin{align}\label{eq:EqualBd}
		 \lvert \Delta_1 \rvert^{1/2} \leq \frac{2}{\pi} \left( \log(n+l) + 2[L : \Q] \log H + 8 \right). 
		 \end{align}

	Assume then that $\Delta_1, \ldots, \Delta_n$ are not all equal. 
	Let $f_1, \ldots, f_n \in \Z_{> 0}$ be such that $\Delta_i = f_i^2 D$ for every $i \in \{1, \ldots, n\}$. 
	Relabelling as necessary, we may assume that $f_1 \geq \ldots \geq f_n$. 
	Since $\Delta_1, \ldots, \Delta_n$ are not all equal, there exists some minimal $k \in \{2, \ldots, n\}$ such that $f_1 > f_k$. 

	For every $i > 1$, Proposition~\ref{prop:jbd} implies that
\[ \lvert x_i \rvert \leq \max \{ e^{\pi f_1 \lvert D \rvert^{1/2}/2} , e^{\pi f_k \lvert D \rvert^{1/2}} \} + 2079.\]
Therefore, \eqref{eq:equal2new} yields that
\begin{align}\label{ineq:samefundzeta1}
e^{\pi f_1 \lvert D \rvert^{1/2}} - 2079 \leq &H^{[L : \Q]} ( (n-1) H^{[L : \Q]} (\max \{ e^{\pi f_1 \lvert D \rvert^{1/2}/2} , e^{\pi f_k \lvert D \rvert^{1/2}} \} + 2079) \nonumber\\
& + (l+1) H^{[L: \Q]}) \nonumber\\
\leq & H^{2[L : \Q]} (n+l) (\max \{ e^{\pi f_1 \lvert D \rvert^{1/2}/2} , e^{\pi f_k \lvert D \rvert^{1/2}} \} + 2079).
\end{align}
If $f_k \leq f_1 / 2$, then \eqref{ineq:samefundzeta1} implies \eqref{ineq:samedisczeta}, and so we obtain again that
	\begin{align}\label{eq:bd11} 
		 \lvert \Delta_1 \rvert^{1/2} \leq  \frac{2}{\pi} \left( \log(n+l) + 2 [L : \Q] \log H + 8 \right).
		 \end{align}
	So we may assume that $f_k \geq f_1/2$, and hence
	\begin{align}\label{eq:goodbd1}
		e^{\pi f_1 \lvert D \rvert^{1/2}} - 2079 \leq H^{2[L : \Q]} (n+l) (e^{\pi f_k \lvert D \rvert^{1/2}} + 2079).
		\end{align}

	The proof now proceeds as in Step~3 of \cite[\S3]{BiluKuhne20}.	
	Suppose first that
	\begin{align*}
		\pi(f_1 - f_k) \lvert D \rvert^{1/2} \geq 2 [L : \Q] \log H + \log(2(n + l)).
	\end{align*}
	Then we obtain from \eqref{eq:goodbd1} that
	\begin{align*}
		&e^{\pi f_1 \lvert D \rvert^{1/2}} - 2079 \leq H^{2[L : \Q]} (n+l) (e^{\pi f_1 \lvert D \rvert^{1/2}} \frac{H^{-2 [L : \Q]}}{2 (n + l)} + 2079),
	\end{align*}
and hence
\[ e^{\pi f_1 \lvert D \rvert^{1/2}} \leq 4158((n+l) H^{2 [L : \Q]} + 1) \leq 4158 H^{2 [L : \Q]} (n+l+1).\]
Therefore,
	\begin{align}\label{eq:bd12}
		 \lvert \Delta_1 \rvert^{1/2} \leq \frac{1}{\pi} \left(2 [ L : \Q] \log H  + \log(n+l+1) + 9 \right) .
	\end{align} 
	
	So we may subsequently assume that
	\begin{align*}
		\pi(f_1 - f_k) \lvert D \rvert^{1/2} < 2 [L : \Q] \log H + \log(2(n+l)).
	\end{align*}
	This implies, as in \cite[(31)]{BiluKuhne20}, that
	\[ \frac{\mathrm{lcm}(f_1, f_k)}{f_k} > \frac{\pi  \lvert \Delta_1 \rvert^{1/2}}{2 [L : \Q] \log H + \log(2(n+l))}.\]
	Let $K = \Q(\sqrt{D})$. Let
	\[ G= \Gal(L \cdot K(x_1) \cdot K(x_k) / L \cdot K(x_k)).\]
	Then, by the same argument as around \cite[(35)]{BiluKuhne20} and appealing to Lemma~\ref{lem:samediscfield}, either there exists an element $\sigma \in G$ which acts non-trivially on 
	\[ a_1 x_1^m + \ldots + a_{k-1} x_{k-1}^m,\]
	or else
	\begin{align}\label{eq:bd13}
		\lvert \Delta_1 \rvert^{1/2} 
		&\leq 138 n^2 [L : \Q]^3  \log H + 69 n^2 \log(2(n+l))[L : \Q]^2,
	\end{align}
	in which case we are done.
	
	We may thus assume subsequently that an element $\sigma \in G$ which acts non-trivially on 
	\[ a_1 x_1^m + \ldots + a_{k-1} x_{k-1}^m\]
	exists. 
	Let $\hat{\sigma} \in \Gal(\alg / L \cdot K(x_k))$ be a corresponding lift of $\sigma$.
	Apply $\hat{\sigma}$ to \eqref{eq:MannEq} and subtract the result from \eqref{eq:MannEq}. 
	One thereby obtains an equation
	\begin{align*}
		\beta_1 y_1^m + \ldots + \beta_s y_s^m + \gamma_1 \omega_1 + \ldots + \gamma_r \omega_r = \delta,
	\end{align*}
	where $y_1, \ldots, y_s$ are pairwise distinct singular moduli for some $s \leq 2n$ and $\omega_1, \ldots, \omega_r$ are pairwise distinct roots of unity for some $r \leq 2l$.
	Let $S$ denote the set of discriminants of $y_1, \ldots, y_s$.
	Then
	\[ \Delta_1 \in S \subset \{ \Delta_1, \ldots, \Delta_n\} \setminus \{\Delta_k\},\]
	since $\hat{\sigma}$ acts non-trivially on $a_1 x_1^m + \ldots + a_{k-1} x_{k-1}^m$ and trivially on $K(x_k)$. 
	Further, $\beta_1, \ldots, \beta_s \in L$ and
	\[ \exp h(\beta_1, \ldots, \beta_s, \gamma_1, \ldots, \gamma_r, \delta) \leq 2 H^2\]
	by \cite[Proposition~1.5.15]{BombieriGubler06}.
Note that
	\[ \# S \leq \# \left( \{f_1, \ldots, f_n \} \setminus \{f_k\} \right) \leq n-1.\]
	
	Repeat this process. 
	At each stage, we either obtain an equation for which the above argument yields a bound on $\lvert \Delta_1 \rvert$, or we reduce by (at least) one the number of distinct discriminants $\Delta_i$ that occur in our equation. 
	After at most $n - 1$ iterations, this process will yield a bound on $\lvert \Delta_1 \rvert$.
	Observe that the right-hand sides of \eqref{eq:EqualBd}, \eqref{eq:bd11}, \eqref{eq:bd12}, and \eqref{eq:bd13} are majorised by
	\[138n^2 [L : \Q]^3 \log H + 69 n^2 [L : \Q]^2 \log(2(n + l) + 1) +6.\]
	Therefore, after at most $n-1$ iterations, we always obtain that
	\begin{align*}
		 \max_{1 \leq i \leq n} \lvert \Delta_i \rvert^{1/2} \leq &138 (2^{n-1}n)^2 [L : \Q]^3 \log (2^{2^{n-1}-1} H^{2^{n-1}})\\
		 & + 69 (2^{n-1}n)^2 [ L : \Q]^2 \log (2 (2^{n-1}n + 2^{n-1}l) + 1) +6\\
		 \leq &138 n^2 8^n [L : \Q]^3 (\log H + \log 2)\\
		 &+  69 n^3 4^n [L : \Q]^2 \log(n + l + 1). \qedhere
		 \end{align*}
\end{proof}

\subsection{A uniform bound on the singular moduli}\label{subsec:unif}

Roots of unity are algebraic integers of height $0$.
Therefore, we may deduce a uniform version of Proposition~\ref{prop:equalfundnonunifMann} by a method analogous to the proof of Theorem~\ref{thm:mindep}.

\begin{prop}\label{prop:equalfundunifMann}
	Let $d, n, l \in \Z_{> 0}$. 
	Let $L$ be a normal number field with $[L : \Q] \leq d$.
	Let $x_1, \ldots, x_n$ be pairwise distinct singular moduli of respective discriminants $\Delta_1, \ldots, \Delta_n$, but all of the same fundamental discriminant $D$.
	Let $\zeta_1, \ldots, \zeta_l$ be roots of unity and let $m \in \Z_{> 0}$.
	If $a_1, \ldots, a_n \in L^\times$ and $b_1, \ldots, b_l, c \in L$ are such that
	\[ a_1 x_1^m + \ldots + a_n x_n^m + b_1 \zeta_1 + \ldots + b_l \zeta_l = c,\]
	then
	\[ \max_{1 \leq i \leq n} \lvert \Delta_i \rvert^{1/2} \leq \exp(10^{34}  2^{ 17 n} ( l + 1)^8 d^{12} m^4 ).\]
\end{prop}

\begin{proof}
Assume that $\lvert \Delta_1 \rvert$ is maximal.
Since $a_1 \neq 0$, there exist subsets
\[ I \subset \{1, \ldots, n\} \mbox{ and } J \subset \{1, \ldots, l\}\]
such that $1 \in I$ and one of either
\[ \{x_i^m : i \in I\} \cup \{\zeta_j : j \in J\} \mbox{ or }  \{x_i^m : i \in I\} \cup \{\zeta_j : j \in J\} \cup \{1\}\]
is minimally linearly dependent over $L$.
	Therefore, by Proposition~\ref{prop:htcoeff}, there exist $a_i' \in L^\times$ for $i \in I$ and $b_j' \in L^\times$ for $j \in J$ and $c \in L$ such that $a_1' = 1$,
	\[ \sum_{i \in I} a_i' x_i^m  + \sum_{j \in J} b_j' \zeta_j = c, \mbox{ and}\]
	\[ h((a_i')_{i \in I}, (b_j')_{j \in J}, c') \leq (n + l)^2 \left(2 \sum_{i=1}^n h(x_i^m) +  \log^+(n+l)\right). \]
 Proposition~\ref{prop:equalfundnonunifMann} implies that
	\[ \lvert \Delta_1 \rvert^{1/2} \leq 138 n^3 8^n d^3 (h((a_i')_{i \in I}, (b_j')_{j \in J}, c') + \log(n+l+1)+1).\]
	
		The $x_i$ are singular moduli, so, by Proposition~\ref{prop:htsing},
	\[ h(x_i^m) \leq  9 \sqrt{2} m \frac{\lvert \Delta_i \rvert^{1/2}}{\cl(\Delta_i)^{1/2}}.\]
	Therefore,
\begin{align}\label{eq:bd100}
	\lvert \Delta_1 \rvert^{1/2} \leq  &138 n^3 8^n d^3 \Biggl( (n+l)^2 \Bigl(18 \sqrt{2} m \sum_{i=1}^n  \frac{\lvert \Delta_i \rvert^{1/2}}{\cl(\Delta_i)^{1/2}} \nonumber\\
	  &+  \log^+(n+l)\Bigr) + \log(n+l+1)+1 \Biggr) \nonumber\\
	   \leq &c_1(m, n, d, l) \sum_{i=1}^n  \frac{\lvert \Delta_i \rvert^{1/2}}{\cl(\Delta_i)^{1/2}} + c_2(n, d, l), 
\end{align}
where
\begin{align*}	
	c_1(m, n, d, l) &= 2484 \sqrt{2} n^3 8^n (n+l)^2 m d^3\\
	 &\leq 10^6 16^n (l+1)^2 d^3 m,
	 \end{align*}
and
\begin{align*}	
	c_2(n, d, l) = &138 n^3 8^n \Bigl(  \left(n+l\right)^2 \log^+\left(n+l\right)+ \log\left(n + l +1\right) + 1\Bigr)d^3\\
	\leq &10^6 16^n (l+1)^3 d^3.
	\end{align*}

If $\cl(\Delta_i)^{1/2} \leq 2 n c_1(m, n, d, l)$, then Proposition~\ref{prop:class} implies that
\begin{align*}
	 \lvert \Delta_i \rvert^{1/2} &\leq \exp((8 \times 10^6) n^4 c_1(m, n, d, l)^4)\\ 
	 &\leq \exp(10^{31}  2^{ 16 n} n^{4} ( l + 1)^8 d^{12} m^4  )\\
	 &\leq \exp(10^{33}  2^{ 17 n} ( l + 1)^8 d^{12} m^4  ).
	 \end{align*}
We therefore obtain from \eqref{eq:bd100} that
\begin{align*}
	 \lvert \Delta_1 \rvert^{1/2} \leq &c_2( n, d, l) 
	+ c_1(m , n, d, l) \sum_{i=1}^n \Big(\frac{ \lvert \Delta_1 \rvert^{1/2}}{2 n c_1(m, n, d, l)}\\
	&+ \exp(10^{33}  2^{ 17 n} ( l + 1)^8 d^{12} m^4 )\Big).
\end{align*}
So 
\[  \lvert \Delta_1 \rvert^{1/2} \leq 2 \left( c_2( n, d, l) + n c_1(m, n, d, l) \exp(10^{33}  2^{ 17 n} ( l + 1)^8 d^{12} m^4 )  \right). \]
Simplifying the constant, we obtain that
\[  \lvert \Delta_1 \rvert^{1/2}  \leq \exp(10^{34}  2^{ 17 n} ( l + 1)^8 d^{12} m^4 ) . \qedhere \]
\end{proof}

\section{Mixed uniform André--Oort}\label{sec:tup2}

Finally, we come to the proof of Theorem~\ref{thm:tups}.
Roughly, our approach is to use Proposition~\ref{prop:equalfundunifMann} and some class field theory to reduce Theorem~\ref{thm:tups} to the cases of either only singular moduli (Corollary~\ref{cor:Mann}) or only roots of unity (proved by Schinzel \cite{Schinzel88}, see also \cite{DvornicichZannier00, Evertse99, Schlickewei96}).

\subsection{Some class field theory}

The following result is certainly known, but we include a proof here for want of a convenient reference.
Denote by $\Q^\mathrm{ab}$ the maximal abelian extension of $\Q$.
In particular, $\Q^\mathrm{ab}$ contains all roots of unity.
For $\Delta < 0$ such that $\Delta \equiv 0, 1 \bmod 4$, let $\cl(\Delta)[2]$ denote the number of elements of order $\leq 2$ in the class group of $\Delta$.

\begin{prop}\label{prop:Qab}
	If $x$ is a singular modulus of discriminant $\Delta$, then
	\[ [\Q^\mathrm{ab}(x) : \Q^\mathrm{ab}] = \frac{\cl(\Delta)}{\cl(\Delta)[2]}.\]
\end{prop}

\begin{proof}
	Let $K = \Q(\sqrt{\Delta})$. 
	Note $K \subset \Q^\mathrm{ab}$. 
	We therefore have the following diagram of Galois extensions of $\Q$:
	\begin{center}
		\begin{tikzcd}
			& \Q^\mathrm{ab}(x)  & \\
			K(x) && \Q^\mathrm{ab} \\
			& K(x) \cap \Q^\mathrm{ab} \\
			& K\\
			& \Q
\arrow[no head, from=2-1, to=1-2]
\arrow[no head, from=2-3, to=1-2]
\arrow[no head, from=3-2, to=2-1]
\arrow[no head, from=3-2, to=2-3]
\arrow[no head, from=3-2, to=4-2]
\arrow[no head, from=4-2, to=5-2]
		\end{tikzcd}
	\end{center}
	By e.g.~\cite[Ch.~VI \S1, Theorem~1.12]{Lang02}, 
	\[ \Gal(\Q^\mathrm{ab}(x) / \Q^\mathrm{ab}) \cong \Gal(K(x) / K(x) \cap \Q^\mathrm{ab}).\]
	Let $G = \Gal(K(x) / \Q)$ and $H = \Gal(K(x) / K)$.
	Then $G$ is a generalised dihedral group, see e.g.~\cite[Lemma~9.3]{Cox22}, and the group $H$ is isomorphic to the class group of discriminant $\Delta$.
	Since $K(x) \cap \Q^\mathrm{ab}$ is the largest abelian extension of $\Q$ contained in $K(x)$, we have that
	\[ \Gal(K(x) / K(x) \cap \Q^\mathrm{ab}) \cong [G, G],\]
	where $[G, G]$ denotes the commutator subgroup of $G$.
	Since $G$ is generalised dihedral, $[G, G] \cong G^2$.
	Every element of $G \setminus H$ is of order $2$, so $\# G^2 = \# H^2$.
	The obvious isomorphism $H^2 \cong H / H[2]$ therefore implies that
	\[[\Q^\mathrm{ab}(x) : \Q^\mathrm{ab}] = \# \Gal(K(x) / K(x) \cap \Q^\mathrm{ab})  = \# (H / H[2]). \qedhere\]
\end{proof}

\begin{prop}\label{prop:2part}
	Let $x$ be a singular modulus of discriminant $\Delta$. 
	Then
	\[ \cl(\Delta)[2] \leq 139 \lvert \Delta \rvert^{1/6}.\]
\end{prop}

\begin{proof}
	By~\cite[Proposition~3.11]{Cox22}, we have that $\cl(\Delta)[2] \leq 2^{ \omega(\lvert \Delta \rvert)}$, where $\omega(\lvert \Delta \rvert)$ denotes the number of distinct prime divisors of $\lvert \Delta \rvert$.
	So $\cl(\Delta)[2] \leq d(\lvert \Delta \rvert)$,	where $d(\lvert \Delta \rvert)$ denotes the number of divisors of $\lvert \Delta \rvert$.
	The stated result then follows from classical bounds for the divisor function $d(n)$.
	
	Indeed, by \cite[(18.1.2)]{HardyWright08}, we have that, for every $n \in \Z_{> 0}$ and $\epsilon > 0$,
	\[\frac{d(n)}{n^\epsilon} \leq \prod_{p \leq 2^{1/\epsilon}} \left( \max_{a \in \Z_{\geq  0} } \frac{a+1}{p^{a \epsilon}} \right). \]
	Now observe that
	\[ \frac{(a+1)+1}{p^{ (a+1) \epsilon}} \leq \frac{a+1}{p^{ a \epsilon}} \mbox{ if and only if } a+1 \geq \frac{1}{p^\epsilon - 1}.\]
	So, denoting $a_p = \left \lfloor 1/({p^\epsilon -1}) \right \rfloor$, we have that
	\[ \frac{d(n)}{n^\epsilon} \leq \prod_{p \leq 2^{1/\epsilon}} \frac{a_p + 1}{p^{a_p \epsilon}}.\]
	Evaluating this product at $\epsilon = 1/6$ gives the stated result.
\end{proof}

\subsection{Bounding the singular moduli}

\begin{prop}\label{prop:MannSing}
	Let $m, d, n, l \in \Z_{> 0}$.
	Let $L$ be a normal number field such that $[L : \Q] \leq d$.
	Let $x_1, \ldots, x_n$ be pairwise distinct singular moduli of respective discriminants $\Delta_1, \ldots, \Delta_n$.
	Let $\zeta_1, \ldots, \zeta_l$ be roots of unity.
If
	\begin{align}\label{eq:singroot}
		 a_1 x_1^m + \ldots +a_n x_n^m + b_1 \zeta_1 + \ldots + b_l \zeta_l = c
		 \end{align}
	for some $a_1, \ldots, a_n \in L^\times$ and $b_1, \ldots, b_l, c \in L$, then
	\[ \max_{1 \leq i \leq n} \lvert \Delta_i \rvert^{1/2} \leq \exp( (l+1)^8 \exp(10^{22} 10^{27n} d^{13} m^4)).\]
\end{prop}

\begin{proof}
	Denote by $D_1, \ldots, D_n$ the respective fundamental discriminants of $x_1, \ldots, x_n$.
	Suppose $i \in \{1, \ldots, n\}$ is such that $D_i \neq D_*$.
	Note that
	\[ [L \cdot \Q^\mathrm{ab}(x_i) :  L \cdot \Q^\mathrm{ab}] \geq \frac{1}{d} [\Q^\mathrm{ab}(x_i) :  \Q^\mathrm{ab}].\]
	So, by Propositions~\ref{prop:Qab} and \ref{prop:2part},
	\[ [L \cdot \Q^\mathrm{ab}(x_i) :  L \cdot \Q^\mathrm{ab}] \geq \frac{1}{d} \frac{\cl(\Delta_i)}{139 \lvert \Delta_i \rvert^{1/6}}.\]
	Since $D_i \neq D_*$, the lower bound for $\cl(\Delta_i)$ in \eqref{eq:Tatnonfund} implies that
\begin{align}\label{eq:bd20}
	 [L \cdot \Q^\mathrm{ab}(x_i) :  L \cdot \Q^\mathrm{ab}] \geq \frac{1}{d} \frac{(7.4 \times 10^{-4}) \lvert \Delta_i \rvert^{5/12}}{139 \lvert \Delta_i \rvert^{1/6}}.
	 \end{align}

In particular, if
\[ \lvert \Delta_i \rvert^{1/2} \geq 10^{11} n^2 d^2,\]
then \eqref{eq:bd20} implies that
\[ [L \cdot \Q^\mathrm{ab}(x_i) :  L \cdot \Q^\mathrm{ab}] > n.\]
In this case, there exists some $\sigma \in \Gal(L \cdot \Q^\mathrm{ab}(x_i) /  L \cdot \Q^\mathrm{ab})$ such that
\[\sigma(x_i) \notin \{x_1, \ldots, x_n\}.\]
Lift $\sigma$ to an element $\hat{\sigma} \in \Gal(\alg / L \cdot \Q^\mathrm{ab})$.
Applying $\hat{\sigma}$ to \eqref{eq:singroot} and subtracting the result from $\eqref{eq:singroot}$, we obtain that
\[ \alpha_1 y_1^m + \ldots + \alpha_s y_s^m = 0\]
for some $s \in \{1, \ldots, 2n\}$ and pairwise distinct singular moduli $y_1, \ldots, y_s$ with $y_1 = \sigma(x_i)$ and $\alpha_1, \ldots, \alpha_s \in L^\times$.
Theorem~\ref{thm:mindep}, applied to a minimally linearly dependent (over $L$) subset of $y_1^m, \ldots, y_s^m$ including $y_1^m$, then gives that
\begin{align}\label{eq:bd21} 
	\lvert \Delta_i \rvert^{1/2} \leq  \exp(10^{20} 10^{26n} d^{12} m^4).
	\end{align}

If $D_* \notin \{D_1, \ldots, D_n\}$, then we are done.
So, without loss of generality, assume that $k \in \{1, \ldots, n\}$ is such that
\[ D_1 = \ldots = D_k = D_*\]
and $D_i \neq D_*$ if $i > k$.
Hence, for every $i \in \{k+1, \ldots, n\}$, 
\[ [\Q(x_i) : \Q] \leq \lvert \Delta_i \rvert^{2/3} \leq \exp\left(\frac{3}{2} 10^{20} 10^{26n} d^{12} m^4\right)\]
by Proposition~\ref{prop:classup} and \eqref{eq:bd21}.
Let
\[ c' = c - \sum_{i > k}^n a_i x_i^m,\]
and observe that
\begin{align}\label{eq:reduceonefund}
	a_1 x_1^m + \ldots +a_k x_k^m + b_1 \zeta_1 + \ldots + b_l \zeta_l = c'.
	\end{align}

Let $L'$ be the compositum of $L$ and all the fields $\Q(\sqrt{\Delta_i}, x_i)$ for $i>k$.
Since every extension $\Q(\sqrt{\Delta_i}, x_i) / \Q$ is Galois \cite[Lemma~9.3]{Cox22}, the extension $L' / \Q$ is normal and
\begin{align*}
	 [L' : \Q] &\leq d \left(2 \exp\left(\frac{3}{2} 10^{20} 10^{26n} d^{12} m^4\right)\right)^n\\
	 &\leq \exp\left( 10^{20} 10^{27n} d^{13} m^4\right).
	 \end{align*}
Now apply Proposition~\ref{prop:equalfundunifMann} to \eqref{eq:reduceonefund}; we obtain that
\begin{align*}
	 \max_{1 \leq i \leq k} \lvert \Delta_i \rvert^{1/2} &\leq \exp(10^{34}  2^{ 17 n} ( l + 1)^8 m^4 \exp\left( 10^{20} 10^{27n} d^{13} m^4\right)^{12}  )\\
	 &\leq \exp( (l+1)^8 \exp(10^{22} 10^{27n} d^{13} m^4)). \qedhere
	 \end{align*}
\end{proof}

\subsection{Bounding the roots of unity}

We will use the following consequence of a result of Evertse \cite{Evertse99} on linear equations in roots of unity.

\begin{prop}\label{prop:Evertse}
	Let $l, d \in \Z_{> 0}$.
	Let $L$ be a number field such that $[L : \Q] \leq d$.
	Let $\zeta_1, \ldots, \zeta_l$ be roots of unity.
	Denote by $N_j$ the order of $\zeta_j$.
	Suppose that $b_1, \ldots, b_l, c \in L^\times$ are such that
	\[ b_1 \zeta_1 + \ldots + b_l \zeta_l = c\]
	and no non-empty proper sub-sum of the left-hand side is equal to zero.
	Then
	\[ \max_{1 \leq j \leq l} N_j \leq 2 d^2 (l+1)^{6(l+1)^2}.\]
\end{prop}

\begin{proof}
	Let $b_1, \ldots, b_l, c \in L^\times$.
	By \cite[Theorem]{Evertse99}, the equation
	\begin{align}\label{eq:Evertse}
		 b_1 \zeta_1 + \ldots + b_l \zeta_l = c
		 \end{align}
	has at most $(l+1)^{3(l+1)^2}$ distinct solutions $(\zeta_1, \ldots, \zeta_l)$ for which $\zeta_1, \ldots, \zeta_l$ are roots of unity and no non-empty proper sub-sum of the left-hand side of \eqref{eq:Evertse} vanishes.
	If $(\zeta_1, \ldots, \zeta_l)$ is such a solution to \eqref{eq:Evertse}, then so are all its Galois conjugates over $L$.
	In particular, for every $j \in \{1, \ldots, l\}$,
	\[  [L(\zeta_j) : L] \leq (l+1)^{3(l+1)^2}.\]
	Let $N_j$ denote the order of $\zeta_j$.
	From the elementary observation that
	\[[ \Q(\zeta_j) : \Q] = \varphi(N_j) \geq \sqrt{\frac{N_j}{2}}, \] 
	we obtain that
	\[ \sqrt{\frac{N_j}{2}} \leq [L : \Q] (l+1)^{3(l+1)^2}. \qedhere \]
\end{proof}

\begin{remark}
	An analogous bound would follow from the results of Dvornicich and Zannier \cite{DvornicichZannier00}.
	We use the result of Evertse on the grounds that it gives a polynomial, rather than exponential, dependence on $[L : \Q]$ in Proposition~\ref{prop:Evertse} (though a worse dependence on $l$).
\end{remark}

We will now conclude the paper by proving our second main result.

\begin{proof}[Proof of Theorem~\ref{thm:tups}]
	Thanks to Corollary~\ref{cor:Mann} and Proposition~\ref{prop:Evertse}, we may assume that $n, l \geq 1$.
	Let $(x_1, \ldots, x_n, \zeta_1, \ldots, \zeta_l) \in \C^{n+l}$ be a non-degenerate $(m, d, n, l)$-tuple, which is witnessed by coefficients $a_1, \ldots, a_n, b_1, \ldots, b_l, c \in \alg^\times$.
	In particular, $[\Q(a_1, \ldots, a_n, b_1, \ldots, b_l, c) : \Q] \leq d$ and
	\begin{align}\label{eq:Mann}
		 a_1 x_1^m + \ldots + a_n x_n^m + b_1 \zeta_1 + \ldots + b_l \zeta_l + c = 0. 
		 \end{align}
	Denote by $\Delta_i$ the discriminant of $x_i$ and by $N_j$ the order of $\zeta_j$.
	Let $L$ be the normal closure of $\Q(a_1, \ldots, a_n, b_1, \ldots, b_l, c) /\Q$.
	Note that $[L : \Q] \leq d!$.
	
	As in the proof of Corollary~\ref{cor:Mann}, let
	\[ I = \{i \in \{1, \ldots, n\} : \forall j < i \, \, x_j \neq x_i\},\]
and, for each $i \in I$, let 
	\[J_i = \{ j \in \{1, \ldots, n\} : x_j = x_i\} \mbox{ and } a_i' = \sum_{j \in J_i} a_j \in L.\]
	Since $(x_1, \ldots, x_n, \zeta_1, \ldots, \zeta_l)$ is a non-degenerate $(m, d, n, l)$-tuple, we have that $a_i' \neq 0$ for every $i \in I$.
	We may rewrite \eqref{eq:Mann} as
	\[ \sum_{i \in I} a_i' x_i^m + b_1 \zeta_1 + \ldots + b_l \zeta_l + c =0.\]
 We may now apply Proposition~\ref{prop:MannSing} to obtain that
	\begin{align}\label{eq:MannBdSing}
	\max_{1 \leq i \leq n} \lvert \Delta_i \rvert^{1/2} =	\max_{i \in I} \lvert \Delta_i \rvert^{1/2} \leq \exp( (l+1)^8 \exp(10^{22} 10^{27n} (d!)^{13} m^4)).
	\end{align}

Let $L'$ be the normal closure of the extension $L(x_1, \ldots, x_n) / L$.
Then
\begin{align*}
	[L' : \Q] &\leq [L : \Q] \prod_{i=1}^n [\Q(\sqrt{\Delta_i}, x_i) : \Q] \nonumber\\ 
	&\leq d! (2 \max_i \lvert \Delta_i \rvert^{2/3})^n\\
	&\leq 2^n d! \exp\left( \frac{3n}{2}(l+1)^8 \exp(10^{22} 10^{27n} (d!)^{13} m^4)\right) \\
	&\leq \exp((l+1)^8 \exp(10^{23} 10^{27n} (d!)^{13} m ^4))
	\end{align*}
by Proposition~\ref{prop:classup} and \eqref{eq:MannBdSing}.
Let $c' = - c - a_1 x_1^m - \ldots - a_n x_n^m \in L'$.
So
\begin{align}\label{eq:roots}
	 c' = b_1 \zeta_1 + \ldots + b_l \zeta_l
	 \end{align}
by \eqref{eq:Mann}.
Note $c' \neq 0$, since no proper sub-sum of $b_1 \zeta_1 + \ldots + b_l \zeta_l + c$ vanishes.
Proposition~\ref{prop:Evertse} applied to \eqref{eq:roots} then implies that
\begin{align*}
	 \max_{1 \leq j \leq l} N_j &\leq 2 (l+1)^{6(l+1)^2} \exp((l+1)^8 \exp(10^{23} 10^{27n} (d!)^{13} m ^4))^2\\
	 &\leq \exp((l+1)^8 \exp(10^{24} 10^{27n} (d!)^{13} m^4)) . \qedhere
	 \end{align*}
\end{proof}


\begin{thebibliography}{ABPM15}
\bibitem[ABPM15]{AllombertBiluMadariaga15}
B.~Allombert, Yu. Bilu, and A.~Pizarro-Madariaga, \emph{C{M}-points on straight
	lines}, Analytic number theory, Springer, Cham, 2015, pp.~1--18.

\bibitem[And98]{Andre98}
Y.~Andr\'{e}, \emph{Finitude des couples d'invariants modulaires singuliers sur
	une courbe alg\'{e}brique plane non modulaire}, J. Reine Angew. Math.
\textbf{505} (1998), 203--208.

\bibitem[BG06]{BombieriGubler06}
E.~Bombieri and W.~Gubler, \emph{Heights in {D}iophantine geometry}, New
Mathematical Monographs, vol.~4, Cambridge University Press, Cambridge, 2006.

\bibitem[BGT26]{BiluGunTron26}
Yu. Bilu, S.~Gun, and E.~Tron, \emph{Effective multiplicative independence of
	three singular moduli}, Algebra Number Theory \textbf{20} (2026), no.~6,
1073--1123.

\bibitem[BHK20]{BiluHabeggerKuhne18}
Yu. Bilu, P.~Habegger, and L.~K\"{u}hne, \emph{No singular modulus is a unit},
Int. Math. Res. Not. IMRN (2020), no.~24, 10005--10041.

\bibitem[Bin19]{Binyamini19a}
G.~Binyamini, \emph{Density of algebraic points on {N}oetherian varieties},
Geom. Funct. Anal. \textbf{29} (2019), no.~1, 72--118.

\bibitem[Bin20]{Binyamini19}
G.~Binyamini, \emph{Some effective estimates for {A}ndr\'{e}--{O}ort in
	{$Y(1)^n$}}, J. Reine Angew. Math. \textbf{767} (2020), 17--35, with an
appendix by E. Kowalski.

\bibitem[Bin24]{Binyamini24}
G.~Binyamini, \emph{{L}og-{N}oetherian functions}, preprint,
ar{X}iv:2405.16963v1 (2024).

\bibitem[BJST26]{BinyaminiSchmidtJonesThomas26}
G.~Binyamini, G.~Jones, H.~Schmidt, and M.~Thomas, \emph{An effective
	{P}ila--{W}ilkie theorem for sets definable using {P}faffian functions, with
	some diophantine applications}, J. Eur. Math. Soc. (JEMS) (2026), published
online first.

\bibitem[BK20]{BiluKuhne20}
Yu. Bilu and L.~K\"{u}hne, \emph{Linear {E}quations in {S}ingular {M}oduli},
Int. Math. Res. Not. IMRN (2020), no.~21, 7617--7643.

\bibitem[BLM17]{BiluLucaMasser17}
Yu. Bilu, F.~Luca, and D.~Masser, \emph{Collinear {CM}-points}, Algebra \&
Number Theory \textbf{11} (2017), no.~5, 1047--1087.

\bibitem[BLPM16]{BiluLucaMadariaga16}
Yu. Bilu, F.~Luca, and A.~Pizarro-Madariaga, \emph{Rational products of
	singular moduli}, J. Number Theory \textbf{158} (2016), 397--410.

\bibitem[BMZ13]{BiluMasserZannier13}
Yu. Bilu, D.~Masser, and U.~Zannier, \emph{An effective ``theorem of
	{A}ndr\'{e}'' for {CM}-points on a plane curve}, Math. Proc. Cambridge
Philos. Soc. \textbf{154} (2013), no.~1, 145--152.

\bibitem[Cox22]{Cox22}
D.~Cox, \emph{Primes of the form {$x^2+ny^2$}---{F}ermat, class field theory,
	and complex multiplication}, third ed., AMS Chelsea Publishing, Providence,
RI, 2022, with contributions by R. Lipsett.

\bibitem[DZ00]{DvornicichZannier00}
R.~Dvornicich and U.~Zannier, \emph{On sums of roots of unity}, Monatsh. Math.
\textbf{129} (2000), no.~2, 97--108.

\bibitem[Eve99]{Evertse99}
J.~Evertse, \emph{The number of solutions of linear equations in roots of
	unity}, Acta Arith. \textbf{89} (1999), no.~1, 45--51.

\bibitem[Fow20]{Fowler20}
G.~Fowler, \emph{Triples of singular moduli with rational product}, Int. J.
Number Theory \textbf{16} (2020), no.~10, 2149--2166.

\bibitem[Fow23]{Fowler23}
G.~Fowler, \emph{Equations in three singular moduli: the equal exponent case},
J. Number Theory \textbf{243} (2023), 256--297.

\bibitem[Fow26]{Fowler26}
G.~Fowler, \emph{Some uniform effective results on {A}ndr{\'e}--{O}ort for sums
	of powers in {$\mathbb{C}^n$}}, Math. Proc. Cambridge Philos. Soc.
\textbf{180} (2026), no.~3, 607--641.

\bibitem[Gol76]{Goldfeld76}
D.~Goldfeld, \emph{The class number of quadratic fields and the conjectures of
	{B}irch and {S}winnerton-{D}yer}, Ann. Scuola Norm. Sup. Pisa Cl. Sci. (4)
\textbf{3} (1976), no.~4, 623--663.

\bibitem[GZ86]{GrossZagier86}
B.~Gross and D.~Zagier, \emph{Heegner points and derivatives of {$L$}-series},
Invent. Math. \textbf{84} (1986), no.~2, 225--320.

\bibitem[HW08]{HardyWright08}
G.~Hardy and E.~Wright, \emph{An introduction to the theory of numbers}, sixth
ed., Oxford University Press, Oxford, 2008, Revised by D. R. Heath-Brown and
J. H. Silverman, With a foreword by A. Wiles.

\bibitem[K{\"{u}}h12]{Kuhne12}
L.~K{\"{u}}hne, \emph{An effective result of {A}ndr\'{e}--{O}ort type}, Ann. of
Math. (2) \textbf{176} (2012), no.~1, 651--671.

\bibitem[K{\"{u}}h13]{Kuhne13}
L.~K{\"{u}}hne, \emph{An effective result of {A}ndr\'{e}--{O}ort type {II}},
Acta Arith. \textbf{161} (2013), no.~1, 1--19.

\bibitem[K{\"{u}}h21]{Kuhne21}
L.~K{\"{u}}hne, \emph{Intersection of class fields}, Acta Arith. \textbf{198}
(2021), no.~2, 109--127.

\bibitem[Lan35]{Landau35}
E.~Landau, \emph{Bemerkungen zum {Heilbronnschen} {Satz}}, Acta Arith.
\textbf{1} (1935), 1--18.

\bibitem[Lan02]{Lang02}
S.~Lang, \emph{Algebra}, third ed., Graduate Texts in Mathematics, vol. 211,
Springer-Verlag, New York, 2002.

\bibitem[Lau84]{Laurent84}
M.~Laurent, \emph{\'{E}quations diophantiennes exponentielles}, Invent. Math.
\textbf{78} (1984), no.~2, 299--327.

\bibitem[Li21]{Li21}
Y.~Li, \emph{Singular units and isogenies between {CM} elliptic curves},
Compos. Math. \textbf{157} (2021), no.~5, 1022--1035.

\bibitem[LR19]{LucaRiffaut19}
F.~Luca and A.~Riffaut, \emph{Linear independence of powers of singular moduli
	of degree three}, Bull. Aust. Math. Soc. \textbf{99} (2019), no.~1, 42--50.

\bibitem[Man65]{Mann65}
H.~Mann, \emph{On linear relations between roots of unity}, Mathematika
\textbf{12} (1965), 107--117.

\bibitem[Mar19]{Martinez19}
C.~Mart{\'i}nez, \emph{The number of maximal torsion cosets in subvarieties of
	tori}, J. Reine Angew. Math. \textbf{755} (2019), 103--126.

\bibitem[Pap26]{Papas26}
G.~Papas, \emph{Effective {B}rauer--{S}iegel on some curves in {$Y(1)^n$}},
Math. Ann. \textbf{394} (2026), no.~4, 84.

\bibitem[Pau15]{Paulin15}
R.~Paulin, \emph{An explicit {A}ndr\'e-{O}ort type result for
	{$\mathbb{P}^1(\mathbb{C})\times\mathbb{G}_m(\mathbb{C})$}}, Math. Proc.
Cambridge Philos. Soc. \textbf{159} (2015), no.~1, 153--163.

\bibitem[Pau16]{Paulin16}
R.~Paulin, \emph{An explicit {A}ndr\'{e}--{O}ort type result for
	{$\mathbb{P}^1(\mathbb{C})\times \mathbb{G}_\mathrm{m}(\mathbb{C})$} based on
	logarithmic forms}, Publ. Math. Debrecen \textbf{88} (2016), no.~1-2, 21--33.

\bibitem[Pil11]{Pila11}
J.~Pila, \emph{O-minimality and the {A}ndr\'e--{O}ort conjecture for {$\mathbb{
			C}^n$}}, Ann. of Math. (2) \textbf{173} (2011), no.~3, 1779--1840.

\bibitem[Pil14]{Pila14a}
J.~Pila, \emph{Special point problems with elliptic modular surfaces},
Mathematika \textbf{60} (2014), no.~1, 1--31.

\bibitem[Pil22]{Pila22}
J.~Pila, \emph{Point-counting and the {Z}ilber--{P}ink conjecture}, Cambridge
Tracts in Mathematics, vol. 228, Cambridge University Press, Cambridge, 2022.

\bibitem[PST21]{PilaShankarTsimerman21}
J.~Pila, A.~Shankar, and J.~Tsimerman, \emph{Canonical heights on {S}himura
	varieties and the {A}ndr{\'e}--{O}ort conjecture}, with an appendix by
H.~Esnault and M.~Groechenig, preprint arXiv:2109.08788v4 (2021).

\bibitem[PW06]{PilaWilkie06}
J.~Pila and A.~Wilkie, \emph{The rational points of a definable set}, Duke
Math. J. \textbf{133} (2006), no.~3, 591--616.

\bibitem[PZ08]{PilaZannier08}
J.~Pila and U.~Zannier, \emph{Rational points in periodic analytic sets and the
	{M}anin--{M}umford conjecture}, Atti Accad. Naz. Lincei Rend. Lincei Mat.
Appl. \textbf{19} (2008), no.~2, 149--162.

\bibitem[Rif19]{Riffaut19}
A.~Riffaut, \emph{Equations with powers of singular moduli}, Int. J. Number
Theory \textbf{15} (2019), no.~3, 445--468.

\bibitem[Sca04]{Scanlon04}
T.~Scanlon, \emph{Automatic uniformity}, Int. Math. Res. Not. (2004), no.~62,
3317--3326.

\bibitem[Sch88]{Schinzel88}
A.~Schinzel, \emph{Reducibility of lacunary polynomials. {VIII}}, Acta Arith.
\textbf{50} (1988), no.~1, 91--106.

\bibitem[Sch96]{Schlickewei96}
H.~Schlickewei, \emph{Equations in roots of unity}, Acta Arith. \textbf{76}
(1996), no.~2, 99--108.

\bibitem[Sie35]{Siegel35}
C.~Siegel, \emph{{\"U}ber die {Classenzahl} quadratischer {Zahlk{\"o}rper}},
Acta Arith. \textbf{1} (1935), 83--86.

\bibitem[Tat51]{Tatuzawa51}
T.~Tatuzawa, \emph{On a theorem of {S}iegel}, Jpn. J. Math. \textbf{21} (1951),
163--178.

\bibitem[W{\"u}s14]{Wustholz14}
G.~W{\"u}stholz, \emph{A note on the conjectures of {A}ndr\'e-{O}ort and
	{P}ink}, Bull. Inst. Math. Acad. Sin. (N.S.) \textbf{9} (2014), no.~4,
735--779, with an appendix by {L.} {K}\"uhne.

\bibitem[Zag08]{Zagier08}
D.~Zagier, \emph{Elliptic modular forms and their applications}, The 1-2-3 of
modular forms, Universitext, Springer, Berlin, 2008, pp.~1--103.
		
	\end{thebibliography}

\end{document}